\documentclass[11pt]{amsart}
\usepackage[colorlinks=true, pdfstartview=FitV, linkcolor=blue, 
citecolor=blue]{hyperref}

\usepackage{a4wide,graphicx}
\usepackage{amsmath,amssymb}

\theoremstyle{plain}
\newtheorem{theorem}{Theorem}[section]
\newtheorem{prop}[theorem]{Proposition}
\newtheorem{lemma}[theorem]{Lemma}
\newtheorem{coro}[theorem]{Corollary}
\newtheorem{fact}[theorem]{Fact}

\theoremstyle{definition}
\newtheorem{definition}[theorem]{Definition}
\newtheorem{remark}[theorem]{Remark}

\newcommand{\AAA}{\mathbb{A}}
\newcommand{\ZZ}{\mathbb{Z}\ts}
\newcommand{\QQ}{\mathbb{Q}}
\newcommand{\RR}{\mathbb{R}}
\newcommand{\NN}{\mathbb{N}}
\newcommand{\CC}{\mathbb{C}}
\newcommand{\EE}{\mathbb{E}}
\newcommand{\PP}{\mathbb{P}}
\newcommand{\TT}{\mathbb{T}}

\newcommand{\ii}{\mathrm{i}\ts}
\newcommand{\ee}{\operatorname{e}}
\newcommand{\eps}{\varepsilon}

\newcommand{\supp}{\mathrm{supp}}

\newcommand{\cT}{\mathcal{T}}
\newcommand{\cA}{\mathcal{A}}
\newcommand{\cD}{\mathcal{D}}
\newcommand{\cF}{\mathcal{F}}
\newcommand{\cK}{\mathcal{K}}
\newcommand{\cL}{\mathcal{L}}
\newcommand{\cM}{\mathcal{M}}
\newcommand{\cO}{\mathcal{O}}
\newcommand{\scO}{{\scriptstyle\mathcal{O}}}

\newcommand{\cR}{\mathcal{R}}
\newcommand{\cW}{\mathcal{W}}
\newcommand{\cX}{\mathcal{X}}

\newcommand{\vL}{\varLambda}
\newcommand{\vO}{\varOmega}

\newcommand{\myfrac}[2]{\frac{\raisebox{-2pt}{$#1$}}
      {\raisebox{0.5pt}{$#2$}}}
\newcommand{\bs}{\boldsymbol}

 \newcommand{\lm}{\ensuremath{\lambda\hspace*{-3.5pt}
  \text{\raisebox{1.9pt}{$\scriptstyle \ts\backslash$}}}}
\newcommand{\oplam}{\mbox{\Large $\curlywedge$}}
\newcommand{\smoplam}{\mbox{$\curlywedge$}}

\newcommand{\XX}{\mathbb{X}}
\newcommand{\YY}{\mathbb{Y}}

\newcommand{\dd}{\, \mathrm{d}}

\newcommand{\ts}{\hspace{0.5pt}}
\newcommand{\nts}{\hspace{-0.5pt}}
\newcommand{\exend}{\hfill $\Diamond$}

\DeclareMathOperator{\sinc}{sinc}
\DeclareMathOperator{\dens}{dens}
\DeclareMathOperator{\card}{card}
\DeclareMathOperator{\vol}{vol}
\DeclareMathOperator*{\Conv}{\text{\Huge \raisebox{-3.5pt}{$\ast$}}}

\begin{document}

\title[Diffraction of random substitutions]{Diffraction of 
compatible random substitutions\\[2mm]
 in one dimension}

\author{Michael Baake}

\author{Timo Spindeler}
\address{Fakult\"at f\"ur Mathematik, Universit\"at Bielefeld, \newline
\hspace*{\parindent}Postfach 100131, 33501 Bielefeld, Germany}
\email{$\{$mbaake,tspindel$\}$@math.uni-bielefeld.de}

\author{Nicolae Strungaru}
\address{Department of Mathematical Sciences, MacEwan University, \newline
\hspace*{\parindent}10700 \ts 104 Avenue, Edmonton, AB, Canada T5J 4S2}
\email{strungarun@macewan.ca}

\begin{abstract}
  As a guiding example, the diffraction measure of a random 
  local mixture of the two
  classic Fibonacci substitutions is determined and reanalysed via
  self-similar measures of Hutchinson type, defined by a finite family
  of contractions.  Our revised approach yields explicit formulas for
  the pure point and the absolutely continuous parts, as well as a
  proof for the absence of singular continuous components. This
  approach is then extended to the family of random noble means
  substitutions and, as an example with an underlying $2$-adic
  structure, to a locally randomised version of the period doubling
  chain. As a first step towards a more general approach, we interpret
  our findings in terms of a disintegration over the Kronecker factor,
  which is the maximal equicontinuous factor of a covering model set.
 \end{abstract}



\maketitle

\section{Introduction}
In general, the structure of systems with pure point diffraction is
rather well understood \cite{bm,Q}. Due to recent progress, see
\cite{bbm,bgg,BKM} and references therein, also the situation for
various systems with diffraction spectra of mixed type has
improved. Still, the understanding of mixed spectra in the presence of
entropy is only at its beginning and it is desirable to work out
further concrete examples.  Of particular interest, both theoretically
and for applications to crystallography and physics, are cases that
combine randomness with the presence of long-range aperiodic order,
which means we will be looking for systems with non-trivial point
spectrum in the presence of positive entropy.

In 1989, Godr\`eche and Luck \cite{gl1} introduced a (locally)
randomised extension of the two classic and well-studied Fibonacci
substitutions, which individually, as well as under global mixtures,
define the same hull. This is to say that any $S$-adic type sequence
in these two substitutions will not result in a change of the
hull it defines. In contrast, the \emph{local} mixture, which emerges by
deciding randomly, for each single letter, which of the two substitutions
to apply (see Definition~\ref{def:ransubst} for details),
 results in a much larger hull with positive entropy. Reference \cite{gl1}
contains first results on the topological entropy and the spectral
type of the diffraction measure for the associated point sets. More
precisely, the authors computed the set of Bragg peaks and argued that the
diffraction measure is of mixed type, consisting of a pure point
and an absolutely continuous part. Important steps towards a
proof were given in \cite{mo,mo2}, though the absence of singular
continuous components remained unproved.

The purpose of this paper is to present a closed expression for the
formula of the diffraction intensities, and to give a proof for the
fact that the diffraction measure is indeed of the expected mixed
type, without singular continuous part. Later, we will consider
various generalisations of this example, for instance by regarding the
so-called \emph{noble means families}, each consisting of finitely
many primitive substitution rules that individually all define the
same two-sided discrete dynamical hull, as does their global
mixture. These cases had previously been considered in \cite{mo2}. 
Here, we present a closed expression for the
entire diffraction measures of their \emph{locally}
randomised version. Clearly, several
results are already contained in \cite{mo,diss-T}, the proofs of which will
not be repeated here; see also \cite{mo2} as well as the brief
treatment in \cite[Ch.~11]{BG}. 

As an  interesting generalisation, we also consider a locally  randomised 
version of the \emph{period doubling chain}, which is built from 
a compatible pair of constant-length substitutions. While this is a
simpler situation in the deterministic setting, this is not so here: In fact,
its treatment requires a cut and project scheme with internal group 
$\ZZ_{2}$, the $2$-adic integers, and is thus treated separately.
The type of spectral result we obtain is nevertheless the same:
Almost surely, with respect to the ergodic patch frequency measure
on the stochastic hull, the diffraction measure is of mixed type,
with a pure point part that resembles the deterministic chain and
an absolutely continuous part, but no singular continuous one.
 \smallskip

The paper is organised as follows.  We begin with a brief review
of the deterministic Fibonacci tiling in Section~\ref{sec:fibo}, tailored
to our later needs. Here, we also recall some now classic notions
from the theory of aperiodic order, in particular its central tool, the
cut and project method. To avoid unnecessary repetitions, we
assume the interested reader to consult the recent monograph
\cite{BG} for details.
Section~\ref{sec:ran-fibo} is devoted to the construction and basic
properties of the random Fibonacci inflation, which is followed by an
extension to the family of random noble means inflations in
Section~\ref{sec:ran-nob}. Up to this point, our entire treatment
does not need any abstract tools from the theory of locally compact
Abelian groups.

This changes when we turn our attention to an example
from the class of constant-length substitutions, which we keep
separate in order not to overburden the exposition with a more
abstract setting at the start. Here, we first recall
the basic results for the deterministic period doubling chain in
Section~\ref{sec:pd}, then extending it to a randomised version in
Section~\ref{sec:ran-pd}.  Finally, looking back at the two types of
examples, we identify one common structure 
in Section~\ref{sec:decompose}, where we harvest an interesting
connection between random substitutions and the theory of
iterated functions systems, most notably Elton's ergodic theorem 
\cite{Elton}. As a result, for all random substitutions discussed in this
paper, the topological point spectrum is trivial, but one has a nice
disintegration formula over the Kronecker factor, the latter emerging
as the maximal equicontinuous factor of a covering model set.  Also,
the discontinuous eigenfunctions become continuous on a subset
of the hull of full measure. A brief
outlook concludes our exposition, and is followed by an appendix that
proves some tricky, but often needed, approximation results for the
autocorrelation.

\section{Deterministic Fibonacci tiling}\label{sec:fibo}

Before we investigate random substitutions, let us recall a
paradigmatic deterministic case; see \cite{BG,mo} for a detailed
exposition.  Consider the binary alphabet $\cA=\{a,b\}$ and the
Fibonacci substitution given by
$\zeta^{}_{\mathrm{F}\nts ,1} \! : \, a \mapsto ab$, $b\mapsto a$ or
its variant $\zeta^{}_{\mathrm{F}\nts ,0} \! : \, a \mapsto ba$,
$b\mapsto a$.  Both share the same substitution matrix and also define
the same symbolic hull $\XX$, see \cite[Ex.~4.6 and Rem.~4.6]{BG},
wherefore we call them \emph{compatible}.

Here, we are mainly interested in the \emph{geometric} counterpart
$\YY$ of $\XX$ that emerges as the hull of the Fibonacci tilings, 
where $a$ and $b$ are tiles (intervals in our case) with natural 
lengths.  As in \cite{BG}, we use length
$\tau= \bigl(1+ \mbox{\small $\sqrt{5}$}\,\bigr)/2$ for $a$ and $1$
for $b$.  This gives a topological dynamical system $(\YY,\RR$) under
the translation action of $\RR$ that is strictly ergodic, with pure
point spectrum.  The pure-pointedness equivalently applies to both the
diffraction and the dynamical spectrum \cite{LMS,BL1}. In terms of
diffraction, the spectral properties can be summarised as follows; see
\cite{Hof,BG} and references therein for proofs.

\begin{theorem}\label{theo:a} 
  Let\/ $\YY$ be the geometric hull of the Fibonacci tiling system,
  with prototiles of length\/ $\tau$ for type\/ $a$ and\/ $1$ for
  type\/ $b$. Then, the topological dynamical system\/ $(\YY, \RR)$ is
  strictly ergodic and has pure point dynamical spectrum.

  Now, fix some\/ $\cT \in \YY$ and let\/
  $\vL =\vL_a \ts \dot{\cup}\ts \vL_b$ be the corresponding set of
  left endpoints of the tiles in\/ $\cT$.  Then, the weighted Dirac
  comb\/  $\omega = u^{}_{a} \ts \delta^{}_{\!\vL_a} + \ts u^{}_{b} \ts
  \delta^{}_{\!\vL_b}$,
  with any fixed pair of weights\/ $u^{}_{a},u^{}_{b} \in \CC$, is
  pure point diffractive. Its autocorrelation is given by\/
\[
   \gamma \, = \sum_{z\in\vL-\vL}\, \sum_{\alpha,\beta \in \{a,b\}}
   \!\! \overline{u^{}_{\alpha}} \: \eta^{}_{\alpha \beta} (z) \ts\ts
   u^{}_{\beta} \; \delta_{z} \ts,
\]
where\/ $\eta^{}_{\alpha \beta} (z) = \dens \bigl( \vL^{}_{\alpha} 
\cap (\vL^{}_{\beta} - z)\bigr)$, and the diffraction measure reads
\[
    \widehat{\gamma} \, = \sum_{k\in\ZZ[\tau]/\sqrt 5} 
    I(k)\, \delta_k \ts ,
\] 
 where\/ $I(k) = \bigl| u^{}_a A^{}_{\!\vL_a}(k) + u^{}_b
 A^{}_{\!\vL_b}(k) \bigr|^2$, with the Fourier--Bohr 
 coefficients of\/ $t + \vL_{\alpha}$, for fixed\/ $t\in\RR$ and\/
 $\alpha \in \{ a, b\}$, being given by
\[
   A^{}_{\ts t + \vL_{\alpha}}(k)\, = \lim_{r\to\infty} \myfrac{1}{2r} 
   \sum_{x\in (t+\vL_{\alpha}) \cap [-r,r]} \! \ee^{-2\pi\ii kx}   .
\]
In particular, $\gamma$ and\/ $\vL-\vL$, as well as\/
$\widehat{\gamma}$ and\/ $I(k)$, are independent of\/ $\cT$, while the
Fourier--Bohr coefficients do depend on the chosen element, but
converge uniformly in\/ $t$.  \qed
\end{theorem}

\begin{remark}\label{rem:spec}
  The Fourier--Bohr coefficients in Theorem~\ref{theo:a} exist for all
  $k\in\RR$, but vanish unless
  $k\in\ZZ[\tau] /\mbox{\small $\sqrt{5}$}$. For any $k$ of the latter
  type, the map
  $\vL\mapsto A^{}_{\!\vL}(k) = A^{}_{\!\vL_a}(k) + A^{}_{\!\vL_b}(k)$
  defines a continuous eigenfunction \cite{Hof,Daniel} for
  $(\YY,\RR)$, with
\[
    A^{}_{t+\vL}(k) \, = \, \ee^{-2\pi\ii kt} A^{}_{\!\vL}(k),
\]  
provided $A^{}_{\!\vL}(k) \neq 0$. This connection can be used to show
that the dynamical spectrum is given by
$\ZZ[\tau]/\mbox{\small $\sqrt{5}$}$; see \cite[Sec.~9.4.1]{BG} for
details on extinctions. The latter precisely occur for $k = \ell \tau$
with $\ell \in \ZZ \setminus \{ 0 \}$ provided that
$u^{}_{a} = u^{}_{b}$. For a generic choice of the weights, no
extinctions are present.  \exend
\end{remark}

Via the projection method, compare \cite[Sec.~7]{BG}, the elements of
$\YY$ can be described as (translations of) regular \emph{model sets}
within the cut and project scheme, or CPS for short, $(\RR,\RR,\cL)$.
This is briefly summarised by the diagram
\begin{equation}\label{eq:CPS}
\renewcommand{\arraystretch}{1.2}\begin{array}{r@{}ccccc@{}l}
   & \RR & \xleftarrow{\,\;\;\pi\;\;\,} & \RR \times \RR & 
        \xrightarrow{\;\pi^{}_{\mathrm{int}\;}} & \RR & \\
  \raisebox{1pt}{\text{\footnotesize dense}\,} \hspace*{-1ex}
   & \cup & & \cup & & \cup & \hspace*{-1ex} 
   \raisebox{1pt}{\,\text{\footnotesize dense}} \\
   & \ZZ[\tau] & \xleftarrow{\; 1-1 \;} & \cL & 
        \xrightarrow{\; 1-1 \;} &\ZZ[\tau] & \\
   & \| & & & & \| & \\
   & L & \multicolumn{3}{c}{\xrightarrow{\qquad\qquad\;\,\star
       \,\;\qquad\qquad}} 
       &  {L_{}}^{\star\nts}  & \\
\end{array}\renewcommand{\arraystretch}{1}
\end{equation}
where the lattice, $\cL$, is given by
\[
   \cL \, := \, \bigl\{(x,x^{\star}) \mid x\in \ZZ[\tau] \bigr\}
\] 
and the $\star$-map is algebraic conjugation in
$\QQ(\mbox{\small $\sqrt{5}$} \,)$, as defined by the unique extension
of $\mbox{\small $\sqrt{5}$} \mapsto -\mbox{\small $\sqrt{5}$}$ to a
field automorphism of $\QQ(\mbox{\small $\sqrt{5}$} \,)$.  Concretely,
consider the fixed point of $\zeta_{\mathrm{F}\nts, 1}^2$ with legal
seed $b|a$. The corresponding tiling $\cT$ leads to
$\vL=\vL_a \ts \dot{\cup} \ts \vL_b$ together with
\begin{equation}\label{eq:model-1}
   \vL_a \, = \oplam\bigl([\tau-2,\tau-1 ) \bigr) \, , \quad
   \vL_b \, = \oplam\bigl([-1,\tau-2 ) \bigr)   
   \quad \text{and} \quad 
   \vL \, = \oplam\bigl( [-1,\tau-1) \bigr) ,
\end{equation}
where $\oplam(W) := \{ x \in \ZZ[\tau] \mid x^{\star} \in W \}$; see
\cite[Ex.~7.3]{BG} for details, in particular for the changes in the
windows if we would work with the other possible seed, $a|a$.  In the
formulation of Theorem~\ref{theo:a}, the Fourier--Bohr coefficients of
this particular $\vL$, for $k \in \ZZ [\tau]/\sqrt{5}$, read
\begin{equation}\label{eq:FB-1}
   A^{}_{\!\vL_{a}}(k) \, = \, \myfrac{1}{\mbox{\small $\sqrt{5}$}} 
   \int_{\tau-2}^{\tau-1} \ee^{2\pi\ii k^{\star} y} \dd y 
   \quad \ts \text{and} \quad
   A^{}_{\!\vL_{b}}(k) \, = \, \myfrac{1}{\mbox{\small $\sqrt 5$}} 
   \int_{-1}^{\tau-2} \ee^{2\pi\ii k^{\star}y} \dd y. 
\end{equation}
When working with the other inflation rule, based on
$\zeta_{\mathrm{F}\nts, 0}^2$, completely analogous formulas
can be derived; see Remark~\ref{rem:limit-cases} below for 
further comments and \cite{BG} for the general theory.

For the rule we selected here, similar expressions 
can also be derived for other elements of $\YY$; see
\cite{BG,mo} for more. In this setting, one can express the pair
correlation coefficients $\eta^{}_{\alpha\beta}(z)$ from
Theorem~\ref{theo:a} as
\[
   \eta^{}_{\alpha\beta}(z) \, = \, \dens (\vL) \,
   \frac{\vol \bigl(W^{}_{\!\alpha} \cap 
   (W^{}_{\!\nts \beta} - z^{\star}) \bigr)}{\vol (W)}
   \, = \, \myfrac{1}{\mbox{\small $\sqrt{5}$}} 
   \int_{\RR} 1^{}_{W^{}_{\!\alpha}} (y)
   \, 1^{}_{W^{}_{\!\nts \beta} - z^{\star}} (y) \dd y \ts ,
\]
where the $W^{}_{\!\alpha}$ with $\alpha\in\{a,b\}$ are the \emph{windows}
for the model set description from Eq.~\eqref{eq:model-1}.

Due to the Pisot nature of the golden ratio, $\tau$, one can go one step
further and consider a modified (or deformed) 
hull $\tilde{\YY}$ that emerges from $\XX$ by
taking $a$ and $b$ type intervals of lengths
\begin{equation}  \label{eq:tile-length}
   \ell_{a} \, = \, \tau + \rho \ts (1- \tau) \quad\text{and}\quad 
   \ell_{b} \, = \, 1+\rho \ts ,
\end{equation}
respectively, where $\rho \in (-1, \tau+1)$ is a real parameter. This
choice is made so that the average tile length, and hence also the
density of left endpoints, is the same for $\YY$ and $\tilde{\YY}$. Since the
elements of $\tilde{\YY}$ are always considered as tilings with two distinct
prototiles, even if they have the same length (as happens for
$\rho = \tau^{-2} \ts $), the dynamical systems $(\YY,\RR)$ and
$(\tilde{\YY},\RR)$ are topologically conjugate in this setting. This follows
from the description of the elements of $\tilde{\YY}$ as \emph{deformed model
  sets}, see \cite{BD,BL2}, and is in line with the general analysis
of \cite{CS}.  In terms of the diffraction, and in complete analogy to
\cite[Ex.~9.9]{BG}, the result is the following.

\begin{coro}
  Consider the dynamical system\/ $(\tilde{\YY},\RR)$ with parameter\/
  $\rho \in (-1, \tau+1)$ as above. Select any tiling\/
  $\cT^{\ts \prime} \in \tilde{\YY}$ and consider the corresponding point
  set\/ $\vL ' = {\vL}_{a} ' \ts \dot{\cup}\ts {\vL}_{b} '$ of left
  endpoints.  Then, the Dirac comb\/
  $\omega' = u^{}_{a} \ts \delta^{}_{\!\vL_{a}'} + u^{}_{b}\ts
  \delta^{}_{\!\vL_{b}'}$
  is pure point diffractive, with diffraction measure
\[
    \widehat{\,\gamma^{\ts\prime}\,} \, = \sum_{k\in\ZZ[\tau]/\sqrt 5} 
    I^{\ts\prime}(k)\, \delta_k \ts ,
\] 
where
$I^{\ts\prime}(k) = \bigl| u^{}_{a} A^{\prime}_{\!\vL_a}(k) + u^{}_{b}
A^{\prime}_{\!\vL_b}(k) \bigr|^2$,
and the Fourier--Bohr coefficients are defined in complete analogy to
Theorem~\emph{\ref{theo:a}}.  As before, the diffraction measure is
independent of\/ $\cT^{\ts \prime}$, while the Fourier--Bohr
coefficients do depend on it, but converge uniformly.  \qed
\end{coro}

If we select $\cT^{\ts\prime}$ as the deformed version of the fixed
point tiling from Eq.~\eqref{eq:model-1}, we obtain an analogue of
Eq.~\eqref{eq:FB-1} for the Fourier--Bohr coefficients. With
$\sinc (x) := \frac{\sin (x)}{x}$, one finds
\begin{align*}
   A^{\prime}_{\!\vL_a}(k) \, & = \, 
   \myfrac{1}{\mbox{\small $\sqrt{5}$}} 
   \int_{\tau-2}^{\tau-1} \ee^{2\pi\ii ( k^{\star} - \rho \ts k) y} \dd y 
   \, = \, \myfrac{\ee^{\pi \ii (2\tau-3)(k^{\star}-\rho\ts k)}}
     {\mbox{\small $\sqrt{5}$}}\ts
   \sinc \bigl( \pi  (k^{\star} - \rho \ts k)\bigr) \\
\intertext{and}
   A^{\prime}_{\!\vL_b}(k) \, & =  \, 
   \myfrac{1}{\mbox{\small $\sqrt{5}$}} 
   \int_{-1}^{\tau-2} \ee^{2\pi\ii ( k^{\star} - \rho \ts k) y} \dd y 
   \, = \, \myfrac{\ee^{\pi \ii (\tau-3)(k^{\star}-\rho\ts k)}}
      {\tau \mbox{\small $\sqrt{5}$}} \ts
   \sinc \bigl( \pi (\tau - 1) (k^{\star} - \rho \ts k)\bigr) ,
\end{align*}
which sum to
\[
   A^{\prime}_{\!\vL} (k) \, = \, A^{\prime}_{\!\vL_a}(k)
   + A^{\prime}_{\!\vL_b}(k) \, =  \, \myfrac{\tau 
   \ee^{-\pi\ii \tau^{-2} (k^{\star} - \rho\ts k)}}
    {\mbox{\small $\sqrt{5}$}} \ts
   \sinc \bigl( \pi \tau (k^{\star} - \rho \ts k)\bigr) .
\]
For $\rho = 0$, this gives back the previous expressions.

\begin{remark}
  Our derivation employed the method of deformed model sets, see
  \cite{BD} for details, but the special cases at hand can
  alternatively be written as model sets with a new lattice that
  emerges from the original one by a shear in the physical direction.

  Moreover, due to our choice of tile lengths according to
  Eq.~\eqref{eq:tile-length}, the dynamical spectrum remains pure
  point and is always given by $\ZZ[\tau]/\mbox{\small $\sqrt{5}$}$,
  as in Remark~\ref{rem:spec}.  In particular, all deformed model sets
  in our class here define dynamical systems that are metrically
  isomorphic by the Halmos--von Neumann theorem, and are even
  topologically conjugate to each other as mentioned earlier.  \exend
\end{remark}

\section{Random Fibonacci tiling}\label{sec:ran-fibo}

The topological dynamical system $(\XX,S)$ from the previous section
has zero entropy, because the word complexity is linear.
Alternatively, this fact also follows from \cite{blr}. Let us now
generalise the Fibonacci substitution and construct random Fibonacci
sets with positive entropy that still show long-range order
\cite{gl1,Claudia,nil,mo2}.

To this end, let $p\in[0,1]$ be a fixed probability and set
$q=1-p$. If $\cA = \{ a, b \}$ is our binary alphabet as before, we
use $\cA^*$ to denote the set of finite words with letters from
$\cA$. This is a monoid under the concatenation of words as
multiplication. Now, the \emph{random Fibonacci substitution} is the
endomorphism $\zeta^{}_{\mathrm{F}} \! : \, \cA^*\to \cA^*$ defined by
\begin{equation}\label{eq:Fib-def}
  \zeta^{}_{\mathrm{F}}  : \; 
  \begin{cases}
      a\mapsto 
      \begin{cases}
      ba, & \text{with probability } p,  \\
      ab, & \text{with probability } q,
      \end{cases} \\
      b\mapsto a \ts .
  \end{cases} 
\end{equation}
Here, given a word $w \in \cA^*$, we \emph{independently} apply
$\zeta^{}_{\mathrm{F}}$ to each letter of $w$, which reflects the
endomorphism property.  Consequently, $\zeta^{}_{\mathrm{F}}(a)$ as
well as $\zeta^{}_{\mathrm{F}}(b)$, and hence
$\zeta^{}_{\mathrm{F}}(w)$, have to be considered as random variables,
each with finitely many possible realisations in our case. A more
general definition will be given in the next section, but is not
needed here.

Of course, the term `random' is only justified for $p\in (0,1)$, while
$p=0$ and $p=1$ correspond to the two deterministic limiting cases
from the previous section. The latter are primitive, with the same
substitution matrix, which is inherited by the random version.  To
proceed, we need to adjust some definitions from symbolic dynamics to
the stochastic situation, where we follow the approach of
\cite{mo,mo2} which extends earlier work \cite{gl1,Claudia}; see also
\cite{RS,PG}.

\begin{definition}\label{def:hull}
  Let $\zeta^{}_{\mathrm{F}}$ be the primitive random substitution
  from Eq.~\eqref{eq:Fib-def}.  Then, for any $v,w\in\cA^{*}$ and
  $k\in\NN$, we use $v\blacktriangleleft \zeta^{k}_{\mathrm{F}} (w)$
  to express that $v$ is a subword of at least one realisation of
  $\zeta^{k}_{\mathrm{F}} (w)$.  A word $w\in \cA^{*}$ is called
  \emph{$\zeta^{}_{\mathrm{F}}$-legal} if there is a $k\in\NN$ such
  that $w\blacktriangleleft \zeta^{k}_{\mathrm{F}} (a)$. The
  \emph{$\zeta^{}_{\mathrm{F}}$-dictionary} is defined as
\[
   \cD_{\zeta^{}_{\mathrm{F}}} \, := \, 
   \{w\in\cA^{*}\ |\ w \text{ is $\zeta^{}_{\mathrm{F}}$-legal} \} \ts ,
\]
while the \emph{two-sided discrete stochastic hull} of
$\zeta^{}_{\mathrm{F}}$ is denoted as
\[
   \XX_{\zeta^{}_{\mathrm{F}}} \, := \, \{ w\in\cA^{\ZZ} \mid
   \mathfrak{F}(\{w\}) \subseteq \cD_{\zeta^{}_{\mathrm{F}}}\} \ts ,
\]   
where $\mathfrak{F}(\{w\})$ is the set of all finite subwords of $w$.
\end{definition} 

Obviously, one can replace $\zeta^{}_{\mathrm{F}}$ by any other
primitive random substitution $\varrho$; see
Definition~\ref{def:ransubst} below.  Let us recall a result on the
entropy of the discrete stochastic hull, which illustrates a
fundamental difference from the deterministic case.

\begin{fact}[\cite{nil,mo}]\label{fact:fib-entropy}
  For\/ $\zeta^{}_{\mathrm{F}}$ from Eq.~\eqref{eq:Fib-def} and any\/
  $p\in (0,1)$, the topological entropy of the dynamical system\/
  $(\XX_{\zeta^{}_{\mathrm{F}}}, \ZZ)$ is given by\/
  $\, s=\sum_{\ell=2}^{\infty} \frac{\log(\ell)}{\tau^{\ell+2}} \approx
  0.444 \ts 398 \ts 725$.  \qed
\end{fact}

\subsection{Tiling picture}
To continue, we need to consider the geometric counterpart of
$\XX_{\zeta^{}_{\mathrm{F}}}$.  In a first step, we again replace
letters by intervals of length $\tau$ (for $a$) and $1$ (for $b$), and
take the left endpoints of all intervals to obtain Delone sets. From
now on, we will identify the tiling picture with the Delone picture
that emerges from taking the left endpoints, possibly marked by the
tile type, as this gives a bijection between the two sets of objects.
This procedure turns each element $w\in\XX_{\zeta^{}_{\mathrm{F}}}$ 
into a point set $\vL_{w}$, where we agree to map the marker (origin) 
of $w$ to $0$. This means that each point set $\vL_{w}$ is a Delone set that
contains $0$, and the collection of these sets constitutes the
\emph{punctured continuous hull} $\YY_{\! 0, \zeta^{}_{\mathrm{F}}}$.

At this stage, there is no translation action of $\RR$ on this hull.
To construct a proper dynamical system, we need the
full continuous hull $\YY_{\!\zeta^{}_{\mathrm{F}}}$, which is
obtained as the smallest collection of Delone sets that contains
$\YY_{\! 0, \zeta^{}_{\mathrm{F}}}$ and is closed under the
(continuous) translation action of $\RR$. Then,
$(\YY_{\!\zeta^{}_{\mathrm{F}}}, \RR)$ is a topological dynamical 
system that is topologically conjugate to the suspension of the 
discrete system with a roof function that reflects the tile lengths. 
Clearly, one has $\YY_{\! 0, \zeta^{}_{\mathrm{F}}} = \{ y \in \YY_{\!
  \zeta^{}_{\mathrm{F}}}\mid 0 \in y \}$.

\begin{figure}
\begin{center}
  \includegraphics[width=0.84\textwidth]{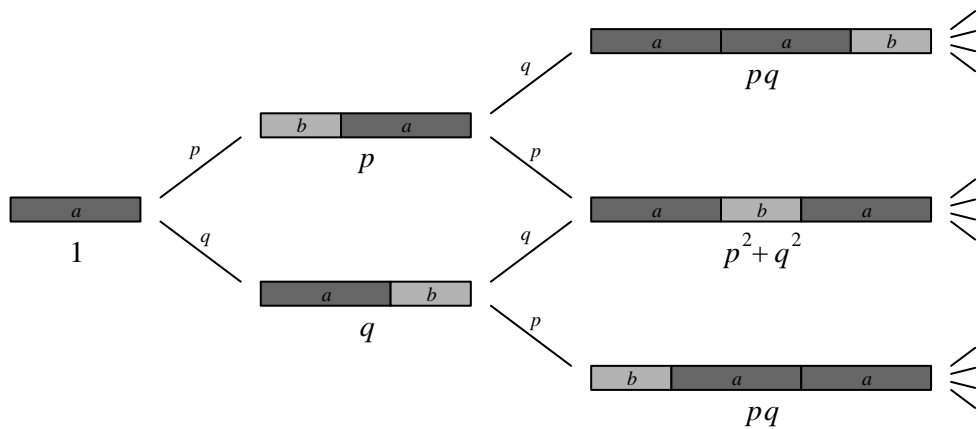}
\end{center}
\caption{\label{fig:infl} Illustration of the first steps of the
  random Fibonacci inflation rule with natural interval lengths. It
  shows the transition probabilities as well as the probabilities for
  the first exact inflation patches.  The probability measure induced
  by this process on the (one-sided) infinite inflation patches is
  compatible with the patch frequency measure obtained from
  Perron--Frobenius theory; compare \cite{mo,PG}. }
\end{figure}

The elements of $\YY^{}_{\!\zeta^{}_{\mathrm{F}}}$ can be described as
subsets of (translates of) model sets. It suffices to see this for the
elements of $\YY_{\! 0, \zeta^{}_{\mathrm{F}}}$. The idea is to use
the positions of the windows from the deterministic case relative to
each other to derive a covering window $W$ with $\vL\subset\oplam(W)$;
see Section~\ref{sec:decompose} for a detailed derivation of this
property.

\begin{remark}\label{rem:generating}
  Recall that, in the deterministic setting, the discrete hull of a
  primitive substitution can be constructed via the fixed point of the
  substitution. In the stochastic situation, there is no direct
  analogue of a fixed point. However, it is possible to modify this
  approach. To do so, define
\[
   X_{\zeta^{}_{\mathrm{F}}}  \, := \, \bigl\{ w\in\cA^{\ZZ} \mid w 
   \text{ is an accumulation point of } 
   \big( \zeta^{k}_{\mathrm{F}}(a |a)\big)_{k\in\NN_0} \bigr\}  ,
\]
where `accumulation point' is meant in the sense of one for any of the
possible realisations of the random substitution sequence.  Then, the
hull $\XX_{\zeta^{}_{\mathrm{F}}}$ from above is the smallest closed
and shift-invariant subset of $\cA^{\ZZ}$ with
$X_{\zeta^{}_{\mathrm{F}}} \subseteq \XX_{\zeta^{}_{\mathrm{F}}}$; see
\cite[Prop.~2.22]{mo}. The geometric realisations of elements of
$X_{\zeta^{}_{\mathrm{F}}}$ are called \emph{generating random
  Fibonacci sets}.  \exend
\end{remark}

\begin{prop}[{\cite[Prop.~5.21]{mo}}]\label{prop:cover}
  Let\/ $\vL$ be any of the generating random Fibonacci sets from 
  Remark~\textnormal{\ref{rem:generating}}. Then, one has\/ 
  $ \vL\subset\oplam(W)$ with
  covering window\/ $W=[-\tau,\tau]$, where\/
  $\dens(\oplam(W)) = \frac{2 \tau}{\sqrt{5}} = 2
  \dens(\vL)$. \qed
\end{prop}

Now, we want to determine the diffraction measure for the system
defined by $\zeta^{}_{\mathrm{F}}$, which means to determine the
average over the diffraction measures of the elements of
$\YY_{\!\zeta^{}_{\mathrm{F}}}$ with respect to an invariant measure
on it.  The latter is chosen as the \emph{patch frequency measure},
denoted by $\nu^{}_{\!\mathsf{pf}}$, with frequencies defined via a
van Hove sequence $(B_n)^{}_{n\in\NN}$ of growing intervals that are
centred at the origin, so $B_n \subset B_{n+1}$ for all $n$ and
$\vol (B_n) \xrightarrow{\, n\to\infty\,} \infty$. Note that
$\nu^{}_{\!\mathsf{pf}}$ is a completely natural choice, and is both
translation invariant and ergodic \cite{mo2}. In fact, it has the nice
feature that the unique measure induced by it on the subset
$X_{\zeta^{}_{\mathrm{F}}}$ via filtration is precisely the measure
defined directly on this subset by the random inflation procedure
according to Figure~\ref{fig:infl}; see \cite{PG} for details.

Consider an individual $\vL\in\YY_{\!\zeta^{}_{\mathrm{F}}}$. Its
diffraction measure (with respect to the same van Hove sequence),
which really is the diffraction measure of the Dirac comb
$\delta^{}_{\!\vL}$, is given by
\begin{equation} \label{eq:expsum}
  \widehat{\ts\gamma^{}_{\!\vL}\ts} \, = 
  \lim_{n\to\infty} \myfrac{1}{\vol (B_n)} 
  \, \Big| \! \sum_{x\in\vL_{n}} \ee^{-2\pi\ii kx} \Big|^2
  \,  = \lim_{n\to\infty} \myfrac{1}{\vol (B_n)} \, 
        \big| X_n(k) \big|^2,
\end{equation}
with $\vL_{n} = \vL \cap B_{n}$ and the exponential sums
$X_n(k):=\sum_{x\in\vL_{n}} \ee^{-2\pi\ii kx}$. The limit is taken in
the vague topology. It exists, and is the same, for
$\nu^{}_{\!\mathsf{pf}}$-almost all
$\vL \in \YY_{\!\zeta^{}_{\mathrm{F}}}$, both as a consequence of the
ergodic theorem.  Let us use $\widehat{\gamma}$ to denote this limit,
which is the diffraction measure of the measure-theoretic dynamical
system $(\YY_{\!\zeta^{}_{\mathrm{F}}},\RR,\nu^{}_{\!\mathsf{pf}})$.
As such, it can be determined as an $\nu^{}_{\mathsf{pf}}$-average
over the individual diffraction measures
$\widehat{\gamma^{}_{\!\vL} \ts }$.

Now, to calculate $\widehat{\gamma}$, we actually do not need to take
an average over (almost) all individual diffraction measures
$\widehat{\ts\gamma^{}_{\!\vL}\ts}$, but only over a suitable subset.
A good choice of the latter will then actually allow us to find the
corresponding \emph{Eberlein decomposition} as well.  A justification
of this step follows later, in Section~\ref{sec:decompose}.  Here, we
work with one-sided tilings (starting from the origin) and approach
them by the family of exact inflation patches, viewed as random
variables. Denoting the corresponding random exponential sums for $n$
inflation steps by $\cX_{n} (k)$, we can determine their probabilistic
weights by the corresponding \emph{random concatenation rule}; see
\cite[Eq.~5.9]{gl1}. This gives the recursion
\begin{equation}\label{eq:concat}
  \cX_n(k) \, = \,
  \begin{cases}
  \cX_{n-2}(k) + \ee^{-2\pi\ii k\tau^{n-2}}\cX_{n-1}(k) \ts , 
     & \text{with probability } p \ts , \\
  \cX_{n-1}(k) + \ee^{-2\pi\ii k\tau^{n-1}}\cX_{n-2}(k) \ts , 
     & \text{with probability } q \ts ,
  \end{cases}
\end{equation} 
together with $\cX_0(k)=1$ and $\cX_1(k)=1$. Note that
$\big(\cX_n(k)\big)_{n\in\NN}$ can be seen as the subsequence of
$\big(X_m (k)\big)_{m\in\NN}$ that corresponds to exact random
Fibonacci inflation patches.

\begin{remark} 
  Here and below, we work with the left endpoints of each tile.
  Consequently, the initial conditions $\cX_0$ and $\cX_1$ differ from
  the ones used in \cite{gl1}, where Godr\`eche and Luck chose the
  right endpoints of each tile instead. This does not affect the final
  result.  \exend
\end{remark}

Let us now use $\EE$ for the average over the exact inflation patches,
weighted with their appropriate probabilities. Then, we obtain
\begin{equation}\label{eq:fibre-mean}
  \EE \bigl(\widehat{\ts\gamma^{}_{\!\vL}\ts}\bigr) \, = \lim_{n\to\infty} 
  \frac{\lvert \EE(\cX_n) \rvert^2}{L_{n}}\, + \lim_{n\to\infty} 
  \frac{\operatorname{Var}(\cX_n)}{L_n}  \, =: \,
  \widehat{\ts\gamma^{}_1 \ts} + \widehat{\ts\gamma^{}_2 \ts},
\end{equation}
where $L_n = \tau^n$ is the length of the level-$n$ random inflation
tiling patch. It is the same for all realisations. The existence of
the two limits was constructively shown in \cite{gl1,mo,mo2}.  It was
also shown there that $\widehat{\ts\gamma^{}_2 \ts}$ is absolutely
continuous. One aim of this paper is to complete the spectral analysis
of this case by showing that $\widehat{\ts\gamma^{}_1 \ts}$ is pure
point, which implies the absence of singular continuous components.

To this end, let $\cM_n$ denote the (finite) random Dirac comb that
underlies $\cX_n$, so $\cX_n = \widehat{\cM_n}$. With initial
conditions $\cM^{}_0 = \cM^{}_1 = \delta^{}_0$, the counterpart to
Eq.~\eqref{eq:concat} reads
\begin{equation}\label{eq:concat-M}
  \cM_n \, = \,
  \begin{cases}
      \cM_{n-2} + \delta_{\tau^{n-2}} * \cM_{n-1} \ts , 
     & \text{with probability } p \ts , \\
       \cM_{n-1} + \delta_{\tau^{n-1}} * \cM_{n-2} \ts , 
     & \text{with probability } q \ts ,
  \end{cases}
\end{equation} 
and we know the following property from \cite{mo,mo2} and
Proposition~\ref{prop:cover}.

\begin{fact}
  For any\/ $n\in\NN$, all realisations of the random Dirac comb\/
  $\cM_n$ have support in the finite point set\/
  $\oplam (W) \cap [0, L_n)$, with\/ $L_n = \tau^n$ as before. \qed
\end{fact}

\subsection{Averages and weight functions}
Define the uniformly discrete point set
\[
   \vL^{}_{\geqslant 0} \, := \, \oplam (W) \cap 
   \RR^{}_{\geqslant 0} \, = \,
   \{ x \in \oplam (W) : x \geqslant 0 \}
\]
and observe that we may identify $\cM_{n}$ with a random variable with
values in $\vO := \{ 0,1 \}^{\vL_{\geqslant 0}}$. Note that $\vO$ is a
compact space, equipped with the standard product topology.  In this
picture, the realisations of $\cM_n$ are sequences of the form
$(m_x )^{}_{x\in\vL_{\geqslant 0}}$ with $m_x \in \{ 0,1 \}$ and
$m_x=0$ for all $x \geqslant L_n$, the latter due to our setting with
left endpoints of the tiles as markers. Let $\vartheta_n$ denote the
corresponding (discrete) probability distribution on $\vO$.

\begin{lemma}\label{lem:conv-distr}
  Let\/ $p$ and\/ $q$ be the probabilities from
  Eq.~\eqref{eq:Fib-def}, and assume that\/ $0 < p <1$. Then, the
  sequence\/ $(\vartheta_n )^{}_{n\in\NN}$ of probability measures
  on\/ $\vO$ is weakly converging.
\end{lemma}

\begin{proof}
  The claimed convergence, by standard arguments, is equivalent to the
  convergence of $\bigl(\vartheta_n ( Z)\bigr)_{n\in\NN}$ for every
  cylinder set $Z$ that is specified at a \emph{finite} set of
  positions in $\vL^{}_{\geqslant 0}$. A simple inclusion-exclusion
  argument shows that this is equivalent to the convergence of
  $\vartheta_n \bigl( \{ m_x = 1 \text{ for } x\in F \} \bigr)$ for
  any \emph{finite} $F\subset \vL^{}_{\geqslant 0}$. Let
\begin{equation}\label{eq:occupation}
    g^{(n)} (x) \, = \, \vartheta_n \bigl( \{ m_x = 1 \} \bigr)
    \, = \, g^{(n)}_{a} (x) + g^{(n)}_{b} (x)
\end{equation}
be the occupation probability of position $x$ under $\vartheta_n$,
split into those for type $a$ and $b$. For any $p\in [0,1]$, the
random inflation \eqref{eq:Fib-def} now implies the recursion
\begin{equation}\label{eq:occupation-prob}
\begin{split}
  g^{(n+1)}_{a} (x) \, & = \, q \, g^{(n)}_{a}
    \bigl( \tfrac{x}{\tau} \bigr) + g^{(n)}_{b}
    \bigl( \tfrac{x}{\tau} \bigr) + p \, g^{(n)}_{a}
    \bigl( \tfrac{x-1}{\tau} \bigr) , \\[2mm]
  g^{(n+1)}_{b} (x) \, & = \, p \, g^{(n)}_{a}
    \bigl( \tfrac{x}{\tau} \bigr) + q \, g^{(n)}_{a}
    \bigl( \tfrac{x-\tau}{\tau} \bigr) ,
\end{split}
\end{equation}
subject to the initial condition
$g^{(0)}_{\alpha} (x) = \delta^{}_{\alpha, a} \, \delta^{}_{x,0}$ and
the general rule that, for any $\alpha \in \{ a,b\}$ and all $n$, we
have $g^{(n)}_{\alpha} (x) = 0$ whenever $x \notin \vL^{}_{\geqslant 0}$.
At $x=0$, this implies
\begin{equation}\label{eq:Markov}
   \begin{pmatrix} g^{(n+1)}_{a} (0) \\
   g^{(n+1)}_{b} (0) \end{pmatrix} \, = \,
   \begin{pmatrix} q & 1 \\ p & 0 \end{pmatrix}
   \begin{pmatrix} g^{(n)}_{a} (0) \\
   g^{(n)}_{b} (0) \end{pmatrix}
\end{equation}
where the matrix on the right-hand side is the (transpose of) a Markov
matrix. It is primitive when $0 < p < 1$, which implies the
convergence of the sequences
$\bigl( g^{(n)}_{\alpha} (0) \bigr)_{n\in\NN}$, with limits
$g^{}_{a} (0) = 1/(1+p)$ and $g^{}_{b} (0) = p/(1+p)$.

Now, writing $\vL^{}_{\geqslant 0} = \{ 0=x^{}_{0} < x^{}_{1} < 
x^{}_{2} < \ldots \}$,
the convergence of $\bigl( g^{(n)}_{\alpha} (x_i)\bigr)_{n\in\NN}$
can be shown inductively in $i$, because the structure of
Eq.~\eqref{eq:occupation-prob} is such that, when $x>0$,
the right-hand side only has arguments $y$ with 
$0\leqslant y < x$, where the functions are either known to
vanish or converge by induction. This shows the convergence
of the marginals for all cylinder sets that are specified 
at a single location.

A similar argument, based on the propagation of prefix probabilities,
also works for the cylinder sets specified on a finite set of
positions. Here, given any finite set
$\varnothing \ne F\subset \vL^{}_{\geqslant 0}$, one chooses an
integer $n$ such that the geometric realisation (as a patch) of any
legal word of length $n$ is longer than the largest element of $F$.
We consider this as a collection of prefix patches.  Taking all
possible inflations of these patches and sorting them according to the
same prefix collection, we derive a transition matrix for the prefix
collection under one inflation step, which is (the transpose of) a
Markov matrix by construction. When $0<p<1$, this matrix is
irreducible by standard arguments and cyclically primitive because the
corresponding graph must contain a loop, which implies primitivity of
the matrix.

Now, we equip the prefix patches with initial probabilities, for
instance via a set of exact inflation patches of sufficient size (we
know that this is possible because we only look at legal words of
length $n$). No matter what these initial probabilities are, a repeated
iteration of the primitive Markov matrix gives a converging sequence
of prefix probability vectors, with the limit being independent of the
initial choice. This finally implies the convergence of
$\vartheta_n \bigl( \{ m_x = 1 \text{ for } x \in F \} \bigr)$ as
$n\to\infty$, and our claim follows.
\end{proof}

\begin{remark}\label{rem:limit-cases}
  It is easy to check that, with our initial condition, the iteration
  of Eq.~\eqref{eq:Markov} also converges for $p=0$, but not for $p=1$,
  where the iteration alternates between the vectors $(1,0)^t$ and
  $(0,1)^t$.  These are the deterministic cases, where the inflation
  has a one-sided fixed point for $p=0$, but not for $p=1$.  In the
  latter case, one has two fixed points for the \emph{square} of the
  inflation instead. Here, one could `restore' convergence by starting
  from the modified initial condition
  $g^{(0)}_{a} (0) = g^{(0)}_{b} (0) = \frac{1}{2}$.

  Looking at the actual sequence of inflation words for $p=1$,
  however, reveals that only the first two positions alternate between
  $ab$ and $ba$, while all following positions are stable. This simply
  reflects the fact that we are working with a \emph{singular} element
  of the discrete hull here, which is not relevant in the random
  situation.
  \exend
\end{remark}

Let us continue with the random case, $0<p<1$, and set
$g^{}_{\alpha} (x) = \lim_{n\to\infty} g^{(n)}_{\alpha} (x)$. Then,
Eq.~\eqref{eq:occupation-prob} implies the exact renormalisation
identities
\begin{equation}\label{eq:consistency}
\begin{split}
  g^{}_{a} (x) \, & = \, q \, g^{}_{a}
    \bigl( \tfrac{x}{\tau} \bigr) + g^{}_{b}
    \bigl( \tfrac{x}{\tau} \bigr) + p \, g^{}_{a}
    \bigl( \tfrac{x-1}{\tau} \bigr) , \\[2mm]
  g^{}_{b} (x) \, & = \, p \, g^{}_{a}
    \bigl( \tfrac{x}{\tau} \bigr) + q \, g^{}_{a}
    \bigl( \tfrac{x-\tau}{\tau} \bigr) ,
\end{split}
\end{equation}
together with $g^{}_{\alpha} (x) = 0$ for $\alpha\in\{ a,b\}$ and
all $x\notin\vL^{}_{\geqslant 0}$. In fact, one has more.

\begin{prop}\label{prop:ren-unique}
  The renormalisation relations from Eq.~\eqref{eq:consistency},
  subject to the condition that\/ $g^{}_{\alpha} (x) = 0$ for\/
  $\alpha\in\{ a,b \}$ and all\/ $x\notin\vL^{}_{\geqslant 0}$, have a
  one-dimensional solution space, for any\/ $p\in [0,1]$. In
  particular, there is precisely one solution with\/
  $g^{}_{a} (0) = 1/(1+p)$.
  
  Moreover, for each\/ $n\in\NN_0$, the support of the
  function\/ $g^{(n)}$ from Eq.~\eqref{eq:occupation} is 
  contained in the finite set\/
  $\vL^{(n)} := \vL \cap [0,L_n)$, and the pointwise limits\/
   $ g(x) = \lim_{n\to\infty} \ts g^{(n)} (x)$ are uniform in the
   sense that\/ $\max_{x\in\vL^{(n)}} \bigl| g(x) - g^{(n)} (x)
   \bigr| \xrightarrow{n\to\infty\ts} 0$.  
\end{prop}

\begin{proof}
  The first claim really is a consequence of
  Lemma~\ref{lem:conv-distr} and its proof. Indeed, taking the limit
  in Eq.~\eqref{eq:Markov} gives an eigenvector equation with
  eigenvalue $1$. It is easy to check that the corresponding
  eigenspace is always one-dimensional, including the cases $p=0$ and
  $p=1$.  Since all other values $g^{}_{\alpha} (x)$ are determined
  recursively, our claim is obvious. The unique solution specified by
  the special value is the one we need for our further analysis.
  
  The supporting set of $g^{(n)} \nts = g^{(n)}_{a} \! + g^{(n)}_{b}$
  follows from the recursion relations in Eq.~\eqref{eq:Markov} by
  induction. The claim on the convergence is obvious for the two
  deterministic cases, $p=0$ and\/ $p=1$, because one then simply has
  $g^{(n)} (x) = g(x)$ on $\vL^{(n)}$.  In general, the concatenation
  structure implies the recursion
\[
   g^{(n+1)} (x) \, = \, \begin{cases}
    q \ts g^{(n)} (x) + p \ts g^{(n-1)} (x) , &
      x \in \vL^{(n-1)} , \\
    q \ts g^{(n)} (x) + p \ts g^{(n-1)} (x \! - \! \tau^{n-1}) , &
      x \in \vL^{(n)}\setminus \vL^{(n-1)} , \\
    p \ts g^{(n)} (x \! - \! \tau^{n-1}) + 
       q \ts g^{(n-1)} (x \! - \! \tau^n ) , &
      x \in \vL^{(n+1)}\setminus\vL^{(n)} . \end{cases}
\]
  This can now be used inductively to show that
\[
    \lim_{n\to\infty} \max_{x\in\vL^{(n)}}
    \bigl| g(x) - g^{(n)} (x) \bigr| \, = \, 0 \ts ,
\] 
the technical details of which are given in \cite[Lemma~3.29]{diss-T}.  
This implies our claim.
\end{proof}

\subsection{Lift into internal space}
Let $\vartheta$ be the probability measure on $\vO$ that is the
weak limit of the sequence $(\vartheta_n )^{}_{n\in\NN}$. This means
that there is a random variable $\cM$ with law $\vartheta$ such that
$(\cM_n)^{}_{n\in\NN}$ converges in distribution to $\cM$, which in
particular implies that
\[
    \EE (\cM) \, = \lim_{n\to\infty} \EE (\cM_n) \, = 
    \sum_{x\in \vL_{\geqslant 0}} \! \vartheta
    \bigl( \{ m_x = 1\} \bigr) \, \delta_x   \ts ,
\]
with
$\vartheta \bigl( \{ m_x = 1\} \bigr) = g^{}_{a} (x) + g^{}_{b} (x)$
as detailed above. Now, we want to determine the coefficients of
$\EE (\cM)$ more explicitly, for which we employ a representation
within our CPS.  To do so, we write $\EE (\cM_n)$ and $\EE (\cM)$ as
\[
    \EE (\cM_n) \; = \sum_{x\in \vL_{\geqslant 0}} 
     \! h^{(n)}(x^\star) \, \delta_{x} \ts 
    \quad \text{and} \quad
    \EE (\cM) \, = \sum_{x\in \vL_{\geqslant 0}}
     \! h (x^\star) \, \delta_{x} \ts .
\]
As suggested by our previous analysis, we write
\[
    h^{(n)} \, = \, h^{(n)}_{a} + h^{(n)}_{b}
\]
with obvious meaning. Define now
$\mu^{(n)}_{\alpha} = \frac{c}{F_n} \sum_{x\in\smoplam(W)}
h^{(n)}_{\alpha} (x^\star) \, \delta^{}_{\nts x^\star}$
for $\alpha\in\{a,b\}$, where $c$ is a positive constant that will be
fixed later and $F_n$ is the number of tiles in the patch that
underlies $\cM_n$, which is a level-$n$ Fibonacci number (with initial
conditions $F^{}_0 = F^{}_1 = 1$ in this case) and independent of the
realisation of $\cM_n$.

Note that the $\mu^{(n)}_{\alpha}$ are finite and positive pure point
measures on the internal space $H=\RR$, all with support in
$W=[-\tau,\tau]$. Moreover, we work in a setting where
$\frac{1}{c} \bigl(\mu^{(n)}_{a} + \mu^{(n)}_{b}\bigr)$ is a
probability measure for each $n\in\NN$. Now, we need a fundamental
property of these measures, which is a reformulation of earlier
results from \cite{Claudia,mo,mo2}. Let $g.\mu := \mu \circ g^{-1}$
denote the push-forward of a measure $\mu$ by a continuous (and, in
our case, always invertible) function $g$. Define affine functions
$g^{}_0$, $f^{}_0$ and $f^{}_{1}$ by $g^{}_0(x)=\sigma {x+1}$ and
$f^{}_j(x)=\sigma(x+j)$, where $\sigma:= \tau^\star = \tau' = 1-\tau$ is the
algebraic conjugate of\/ $\tau$. Then, an explicit calculation similar
to the one from \cite[Eq.~6.28]{mo} shows that we have the recursion
\begin{equation}\label{eq:mu-rec}
  \begin{split}
     \mu^{(n+1)}_a \, &= \, \frac{F_n}{F_{n+1}} \bigl(
       p \ts\ts (g^{}_0 .\mu^{(n)}_a) + q \ts (f^{}_0 .\mu^{(n)}_a)
        + (f^{}_0.\mu^{(n)}_b) \bigl),  \\[1mm]
     \mu^{(n+1)}_b \, &= \, \frac{F_n}{F_{n+1}} \bigl( 
       p \ts\ts (f^{}_0 .\mu^{(n)}_a) + q \ts 
       (f^{}_1 \nts .\mu^{(n)}_a) \bigr),  
  \end{split}
\end{equation}
with initial condition $\mu^{(1)}_{a} = \delta^{}_0$ and
$\mu^{(1)}_{b}=0$.

\begin{lemma}\label{lem:aw}
  The sequences\/ $\bigl(\mu^{(n)}_{\alpha}\bigr)_{n\in\NN}$ of finite
  measures are weakly converging. The limit measures\/
  $\mu^{}_{\alpha}$ are compactly supported and satisfy the system of
  rescaling equations
\[
  \begin{split}
     \mu^{}_a \, &= \, |\sigma|\bigl(
       p \ts\ts (g^{}_0 .\mu^{}_a) + q \ts (f^{}_0 .\mu^{}_a)
        + (f^{}_0.\mu^{}_b) \bigl),  \\[1mm]
     \mu^{}_b \, &= \,  |\sigma| \bigl( 
       p \ts\ts (f^{}_0 .\mu^{}_a) + q \ts (f^{}_1 \nts .\mu^{}_a) \bigr),  
  \end{split}
\]
with\/ $\sigma = \tau'$ and with the functions\/ $g^{}_0$ and\/
$f^{}_j$ as in Eq.~\eqref{eq:mu-rec}.
\end{lemma}

\begin{proof}
  Since the support of any of the measures $\mu^{(n)}_{\alpha}$ is a
  finite subset of the compact set $W\nts$,  the claimed weak
  convergence follows from standard arguments, for instance by taking
  Fourier transforms and applying Levy's continuity theorem; see
  \cite[Thm.~3.14]{BF}.  Clearly,
  the support of the limit measure must also be contained in $W$.

  With $F_n/F_{n+1} \xrightarrow{\,n\to\infty\,} \tau^{-1} =
  \lvert \sigma \rvert$, it also follows that the limit measures
  satisfy the limiting rescaling relations, which are the ones
  stated.
\end{proof}

Let us apply the Fourier transform to the relations of
Lemma~\ref{lem:aw}. Recall that the Fourier transform of a finite
measure is a uniformly continuous function. Moreover, for an affine
function $g$ defined by $g(x) = r x + s$ with $r\ne 0$, one obtains
$\widehat{g \nts .\mu} \ts (t) = \ee^{- 2 \pi \ii s t} \widehat{\mu}
(r t)$;
compare \cite[Cor.~6.24]{mo}. This leads to the contractive relations
\begin{align*}
     \widehat{\mu^{}_a}(t)\, &=\, |\sigma|\,
     \bigl( ( p  \ee^{-2\pi \ii t} + \, q   )
       \, \widehat{\mu^{}_a} (\sigma t)   +  
       \widehat{\ts\mu^{}_b \ts}(\sigma t) \bigr) , \\[1mm]
     \widehat{\ts\mu^{}_b \ts}(t)\, &=\,  |\sigma| \,
      (  p + q \ee^{-2\pi\ii \sigma t} ) \,
     \widehat{\mu^{}_a} (\sigma t) \ts .  
\end{align*}              
One can check that a solution of this system, which is unique up to
multiplication by a constant, is given by
\begin{equation}\label{eq:ft}
  \begin{split}
     \widehat{\mu^{}_a}(t)\, &=\, \tilde{c} \, \ee^{-\pi\ii t} 
     \sinc(\pi t)  \prod_{\ell\geqslant 1} \bigl( p + 
      q\ee^{-2\pi\ii \sigma^\ell t}\bigr),   \\          
     \widehat{\ts\mu^{}_b \ts}(t)\, &=\,  \tilde{c} \, |\sigma| 
     \ee^{-\pi\ii \sigma t} \sinc(\pi\sigma t)  
     \prod_{\ell\geqslant 1} \bigl(p + q\ee^{-2\pi\ii \sigma^\ell t}\bigr),     
  \end{split}                         
\end{equation}
where the constant $\tilde{c}$ is determined by our additional
condition that $\frac{1}{c} \bigl(\mu^{}_a \nts + \mu^{}_b \bigr)$ is
a probability measure, hence
$\frac{1}{c} \bigl( \widehat{\mu^{}_a}(0) + \widehat{\ts\mu^{}_b
  \ts}(0)\bigr) =1$,
which gives $\tilde{c} = c/\tau = c\ts\ts \lvert \sigma \rvert$. For
reasons that will become clear below in Eq.~\eqref{eq:conv-rep}, we
choose $c=\tau$ and thus get $\tilde{c}=1$. \smallskip

\subsection{Continuity}   
As $\widehat{\mu^{}_a} , \widehat{\ts\mu^{}_b\ts}\in L^2(\RR)$,
their (inverse) Fourier transforms are also elements of
$L^2 (\RR)$, and actually of $L^2 (W)$, as they have compact support
within $W$. But $L^2 (W) \subset L^1 (W)$ by H\"{o}lder's
inequality, so the 
measures $\mu^{}_a$ and $\mu^{}_b$ are absolutely continuous with
respect to  Lebesgue measure $\lm$, giving
$\mu^{}_\alpha = h^{}_\alpha \lm$ with Radon--Nikodym densities
$h^{}_\alpha$.  By
standard arguments, compare \cite[Sec.~1.3.4]{Rudin} or
\cite[Thm.~2.2]{Arga}, one has
\[
  \widehat{\ts h^{}_a \ts} \, = \, \widehat{\ts\mu^{}_a\ts} 
  \quad \text{ and } \quad 
  \widehat{\ts\ts h^{}_b\ts\ts} \, = \, \widehat{\ts\mu^{}_b \ts} \ts ,
\] 
where $h^{}_{\alpha}$ should \emph{not} be thought of as a limit of the
sequence $(h^{(n)}_{\alpha})^{}_{n\in\NN}$ in the sense of functions.
We will now show that the $L^{1}$-functions $h^{}_a$ and $h^{}_b$ are
actually represented by \emph{continuous} functions whenever $0<p<1$.
To this end, we employ an old idea of Jessen and Wintner
\cite{JW}. Let us first define
\begin{equation}\label{eq:def-mu}
  \mu \: := \, \Conv_{\ell=1}^{\infty} \bigl( p \, \delta^{}_0 
    + q \ts\ts \delta^{}_{\nts\nts \sigma^\ell} \bigr) \, = \,
    \Conv_{\ell=1}^{\infty} \mu^{}_{\ell} \ts ,
\end{equation}
via the probability measures
$\mu^{}_\ell = p \, \delta^{}_0 + q \ts\ts \delta^{}_{\nts\nts \sigma^\ell}$,
where $p\in [0,1]$ and $q=1-p$. Note that the limiting cases are
$\mu=\delta^{}_{0}$ for $p=1$ and
$\mu=\delta^{}_{-\sigma^2} =\delta^{}_{\tau - 2}$ for $q=1$.
 
\begin{lemma}
  The infinite convolution product for the measure\/ $\mu$ from
  Eq.~\eqref{eq:def-mu} is absolutely convergent to a probability
  measure in the weak topology, which is to say that it is weakly
  convergent to the same limit for any order of the terms.
\end{lemma}

\begin{proof}
  With $M_\kappa (\nu) := \int_{\RR} x^\kappa \dd \nu (x)$ for
  $\kappa \geqslant 0$, we get $M^{}_{0} (\mu^{}_{\ell}) = 1$ and
  $M_\kappa (\mu^{}_{\ell}) = q \ts \sigma^{\kappa \ts \ell}$ for any
  $\kappa > 0$. In particular, the second moments of all 
  $\mu^{}_{\ell}$ clearly exist. Moreover, one has
\[
   \sum_{\ell=1}^{\infty} |M_1 (\mu^{}_\ell )| < \infty 
   \quad \text{and} \quad
   \sum_{\ell=1}^{\infty} M_2(\mu^{}_\ell ) <\infty 
\]
by a standard geometric series argument, because
$\lvert \sigma\rvert = \tau^{-1} < 1$.

Hence, by an application of \cite[Thm.~6]{JW}, the convolution product
is absolutely convergent as claimed. As all $\mu^{}_{\ell}$ are
probability measures, then so is the limit, $\mu$.
\end{proof}

As a consequence, we may rewrite $\mu$ as $ \mu=\mu^{}_j*\nu^{}_j$, 
for any $j\in\NN$, where 
\[
  \nu^{}_j \, := \, \mu^{}_1*\ldots*\mu^{}_{j-1}*\mu^{}_{j+1}*\ldots
  \, = \, \Conv_{i\ne j} \mu^{}_{i} \ts .
\] 
Together with the previous lemma, this has a rather strong
consequence. In fact, the special case $p=\frac{1}{2}$ can be found in
\cite[Sec.~6]{JW}.

\begin{lemma}\label{lem:cont}
  If\/ $0 < p < 1$, the probability measure\/ $\mu$ from
  Eq.~\eqref{eq:def-mu} is continuous.
\end{lemma}

\begin{proof}
  Observe first that, for any $j\in\NN$ and with $\mu^{}_{j} = 
  p \, \delta^{}_{0} + q \ts\ts \delta^{}_{\nts\nts \sigma^{j}}$ as 
  above, one has
\[
  p\ts q \ts\ts \mu^{}_{j} \, \leqslant \, p \, \delta^{}_{\nts\nts \sigma^{j}}
  + q \ts\ts \delta^{}_{0} \, \leqslant \, p \, \delta^{}_{\nts\nts \sigma^{j}}
  + q \ts\ts \delta^{}_{2 \sigma^{j}} + q \ts\ts \delta^{}_{0} +
  p \, \delta^{}_{-\sigma^{j}}  \, = \, \bigl( \delta^{}_{\nts\nts \sigma^{j}}
  + \delta^{}_{-\sigma^{j}}\bigr) * \mu^{}_{j} \ts ,
\]  
which is to be understood as a relation between positive measures.
The convolution with $\nu^{}_{j}$ now leads to the general inequality
\begin{equation}\label{eq:jw-est}
   p \ts q\ts\ts \mu \, \leqslant \, \bigl( \delta^{}_{\nts\nts\sigma^{j}}
     + \delta^{}_{-\sigma^{j}} \bigr) * \mu \ts .
\end{equation}

Assume now, to the contrary of our claim, that there is an element
$x\in\RR$ with $\mu (\{x\}) > 0$.
Since $q=1-p$, we have $p,q\in (0,1)$ and thus $p \ts q > 0$.  Then,
Eq.~\eqref{eq:jw-est}  implies the estimate
\begin{equation}\label{eq:d}
    0 \, < \, p\ts q \, \mu(\{x\}) \, \leqslant \,
    \mu ( \{x-\sigma^j\} ) + \mu ( \{x+\sigma^j \} )\ts .
\end{equation} 
Next, choose $r\in\NN$ with $p \ts q \, \mu(\{x\}) > \frac{1}{r}$,
which is clearly possible, and select $r$ distinct integers,
$j^{}_1 < j^{}_2 < \ldots < j^{}_r$, say. Since $\mu$ is a probability
measure, we obtain
\[
\begin{split}
  1 \; & \geqslant \; \mu  \biggl( 
   \dot{\bigcup_{1\leqslant  s \leqslant  r}} \bigl(
     \big\{x-\sigma^{j_s} \big\}\, \dot{\cup}\,
     \big\{x+\sigma^{j_s} \big\}\bigr) \biggr) \\[1mm]
   & = \, \sum_{s=1}^r \Bigl( \mu \bigl(  
     \big\{x-\sigma^{j_s} \big\}\bigr) +
     \mu \bigl( \big\{x+\sigma^{j_s} \big\} \bigr)  
     \Bigr)    \\[1mm]
   &\overset{\eqref{eq:d}}{\geqslant} 
    \sum_{s=1}^{r} \,  p\ts q\, \mu \bigl(\{x\}\bigr) 
    \; > \: r \, \myfrac{1}{r} \; = \; 1 \ts .
\end{split}             
\]
This contradiction shows that $\mu$ is continuous. 
\end{proof}

\begin{remark}
  Let us mention that one can also show $\mu$ to be a purely singular
  continuous measure.  However, since we only need the continuity
  of $\mu$, we omit this extra step.  \exend
\end{remark}

Before we continue, let us state a classic result that, due to the
lack of a simple reference, we include here with a short proof.

\begin{fact}\label{fact:cont}
  Let\/ $a,b\in\RR$ with\/ $a<b$, and let\/ $\nu$ be a finite, regular
  Borel measure on\/ $\RR$ that is continuous, i.e.,
  $\nu(\{ x \}) = 0$ for all\/ $x\in\RR$. Now, let\/ $J$ be any of the
  intervals\/ $[a,b]$, $(a,b)$, $[a,b)$ or\/ $(a,b]$.  Then, the
  function\/ $1^{}_{\nts J} \nts * \nu $ is continuous on\/ $\RR$.
\end{fact}

\begin{proof}
  Consider the case of a closed interval first.  Fix $x\in\RR$ and
  choose a sequence $(x_n )^{}_{n\in\NN}$ with $x_n \searrow x$ as
  $n\to\infty$. Then, as $\nu$ is a continuous measure by assumption,
  we get
\[
\begin{split}
   \bigl| \bigl(1^{}_{\nts J} \nts * \nu \bigr) (x) -
    \bigl(1^{}_{\nts J} \nts * \nu \bigr) (x^{}_n ) \bigr|
    \, & = \, \Bigl| \int_{\RR} \bigl( 1^{}_{\nts J} (x-y)
    - 1^{}_{\nts J} (x^{}_{n} - y ) \bigr) \dd \nu (y) \Bigr| \\[1mm]
    & = \, \bigl| \nu \bigl( [x-b, x^{}_{n} -b]\bigr) -
     \nu\bigl( [x-a, x^{}_{n} - a] \bigr) \bigr|
    \, \xrightarrow{\, n \to \infty \,}  \, 0
\end{split}
\]
because $\nu(\{ x-b\}) = 0 = \nu(\{ x-a\})$. The analogous relation,
with the same limit, holds when $x_n \nearrow x$. Together, this
implies that the function $1^{}_{\nts J}\nts * \nu$ is continuous on
$\RR$.

Due to the assumed continuity of $\nu$, the same type of argument
applies for half-open or open intervals as well.
\end{proof}

\begin{prop}\label{prop:cont}
  If\/ $0<p<1$, the Radon--Nikodym densities\/ $h^{}_a$ and $h^{}_b$
  are continuous functions with compact support.
\end{prop}

\begin{proof}
  Observe first that
  $\widehat{\mu} (t) = \prod_{\ell\geqslant 1} \bigl( p + q \ts
  \ee^{-2 \pi \ii \sigma^\ell t} \bigr)$
  by an application of the convolution theorem. Then, it follows from
  Eq.~\eqref{eq:ft} that the $L^{1}$-functions $h_a$ and $h_b$ are
  represented by
\begin{equation}\label{eq:conv-rep}
   h^{}_a (x) \, = \, \bigl( 1^{}_{[0,1)} * \mu \bigr)(x)  
   \quad \text{ and } \quad 
   h^{}_b (x) \, = \, \bigl( 1^{}_{[\sigma,0)}
         *\mu\bigr)(x) \ts ,
\end{equation}
with the probability measure $\mu$ from Eq.~\eqref{eq:def-mu}. Here,
we have taken the liberty to choose half-open intervals for the
characteristic functions to match the standard situation in the
limiting cases $p=1$ and $q=1$; compare the discussion in
\cite[Ex.~7.3]{BG}.

Fact~\ref{fact:cont} in conjunction with Lemma~\ref{lem:cont} then
implies the continuity of $h^{}_a$ and $h^{}_b$ for any $0<p<1$.
Moreover, as a consequence of Lemma~\ref{lem:aw}, both functions are
supported on a subset of $W$.
\end{proof}

\begin{remark}
  Let us note that Eq.~\eqref{eq:conv-rep} also holds for the limiting
  cases $p=1$ and $q=1$, where one obtains $h^{}_{a} = 1^{}_{[0,1)}$
  and $h^{}_{b} = 1^{}_{[1-\tau,0)}$ as well as
  $h^{}_{a} = 1^{}_{[\tau-2,\tau-1)}$ and
  $h^{}_{b} = 1^{}_{[-1,\tau-2)}$, respectively. This is in line with
  the description of these deterministic limiting cases as regular
  model sets; compare \cite[Rem.~4.6 and Ex.~7.3]{BG}.  \exend
\end{remark}

Now, for all $x \in \vL^{}_{\geqslant 0}$, we need to relate the
function values $h^{}_{\alpha} (x^\star)$ with the occupation
probabilities $g^{}_{\alpha} (x)$ from above, because the continuity
of the $h^{}_{\alpha}$ is a representation result in the Lebesgue
sense, but $(\vL^{}_{\geqslant 0})^{\star}$ is a null set.

\begin{lemma}
  If\/ $0<p<1$, one has\/ $g^{}_{\alpha} (x) = h^{}_{\alpha} (x^\star)$
  for\/ $\alpha\in\{ a,b\}$ and all\/ $x\in\vL^{}_{\geqslant 0}$.
\end{lemma}

\begin{proof}
  When $0<p<1$, the functions $h_{\alpha}$ are continuous, and satisfy
  the recursions 
\[
\begin{split}
  h^{}_{a} (x^\star ) \, & = \, q \, h^{}_{a}
    \bigl( ( \tfrac{x}{\tau} )^\star \bigr) + h^{}_{b}
    \bigl( ( \tfrac{x}{\tau})^\star \bigr) + p \, h^{}_{a}
    \bigl( ( \tfrac{x-1}{\tau})^\star \bigr) , \\[2mm]
  h^{}_{b} (x^\star ) \, & = \, p \, h^{}_{a}
    \bigl( ( \tfrac{x}{\tau})^\star \bigr) + q \, h^{}_{a}
    \bigl( ( \tfrac{x-\tau}{\tau})^\star \bigr) ,
\end{split}
\]
as a consequence of Lemma~\ref{lem:aw}, rewritten in terms of the
densities. With the initial condition
$h^{}_{a} (0^\star) = 1/(1\nts\nts + \nts p)$, we see that we obtain
the same type of renormalisation equation as in
Eq.~\eqref{eq:consistency}, and an application of
Proposition~\ref{prop:ren-unique} implies our claim.
\end{proof}

\subsection{Diffraction measure of the random Fibonacci hull}
Let us finally come back to the measure $\widehat{\ts\gamma^{}_1\ts}$
from Eq.~\eqref{eq:fibre-mean}. We know from above that 
\[
  \widehat{\ts \gamma^{}_{1}\ts} \, =
  \lim_{n\to\infty} \myfrac{1}{\tau^n}\, \big| \EE(\cX_n) \big|^2 
   = \lim_{n\to\infty} \myfrac{1}{\tau^n}\, \big| 
   \widehat{\EE(\cM_n)} \big|^2  
   =  \lim_{n\to\infty} \myfrac{1}{\tau^n}\, \cF
      \bigl[\EE(\cM_n)*\widetilde{\EE(\cM_n)}\bigr] ,     
\]  
where $\cF$ denotes Fourier transform. By construction,
$\EE(\cM) := \lim_{n\to\infty} \EE(\cM_n)$ is a weighted Dirac comb on
$\oplam([-\tau,\, \tau]) \cap \RR^{}_{\geqslant 0}$, with weights
$g(x) = h (x^\star )$ and the interpretation given previously. 
With the second part of Proposition~\ref{prop:ren-unique}, 
one can now check that 
\[
     \widehat{\ts\gamma^{}_{1}\ts} \, = \,
     \cF \bigl( \EE (\cM) \circledast
     \widetilde{\EE (\cM)}\bigr),
\]     
with $\circledast$ denoting the volume-averaged or Eberlein
convolution of measures on $\RR_{\geqslant 0}$. This holds as a
consequence of the measure
$\bigl| \EE (\cM) |^{}_{[0,L_n)} \! - \EE (\cM_n ) \bigr| =
\sum_{x\in\vL^{(n)}} \ts \lvert g (x) - g^{(n)} (x) \rvert \,
\delta_x$
getting uniformly small on $\vL^{(n)}$ as $n\to\infty$. In fact, one
even has
$\max_{x \in \vL^{(n)}} \bigl| g(x) - g^{(n)} (x) \bigr| \leqslant r^n
$
with $r = \max \{ p,q\} <1$. This implies, as in our earlier case,
that
\[
   \lim_{n\to\infty} \max_{x\in\vL^{(n)}}
   \bigl| g(x) - g^{(n)} (x) \bigr| \, = \, 0 \ts ,
\]
this time with exponentially fast convergence; compare
\cite[Lemma~3.29]{diss-T}.

Next, define the two-sided measure
\begin{equation}\label{dirac}
     \omega^{}_{h} \, = \sum_{x\in\smoplam([-\tau,\, \tau])} 
      h (x^{\star})\, \delta_x \ts ,
\end{equation}
with $h=h^{}_a+h^{}_b$ as before. It is not difficult to check
that one then has
\[
    \widehat{\ts \gamma^{}_{1}\ts} \, = \, \widehat{
    \omega^{}_{h} \! \circledast 
    \widetilde{\omega^{}_{h}}  }
    \, = \, \widehat{ \omega^{}_{\! h * \widetilde{h}} }
\]
where $\circledast$ now denotes the  Eberlein convolution 
for measures on $\RR$; see \cite[Sec.~8.8]{BG}.  Since we have
$h^{}_a,h^{}_b\in C_\mathsf{c}(\RR)$ by Proposition~\ref{prop:cont},
we may now apply the following general result.

\begin{theorem}[{\cite[Thm.~9.5]{BG}\label{theo}}] 
  Consider the weighted Dirac comb
\[
    \omega_g = \sum_{x\in\vL} g(x^{\star})\, \delta_x
\]  
on a regular model set\/ $\vL=\oplam(W)$ with CPS\/ $(\RR^d,H,\cL)$
and compact window\/ $W=\overline{W^{\circ}}\subseteq H$, with a
function\/ $g \! : \, H\xrightarrow{\quad} \CC$ which is continuous
on\/ $W$ and vanishes on its complement. Then, $\omega_g$ has the
positive, translation bounded, pure point diffraction measure
\[
  \widehat{\ts\gamma^{}_{\omega_g}\ts} \: = \sum_{k\in L^{\circledast}} 
   |A(k)|^2\, \delta^{}_k 
   \quad \text{with} \quad 
  A(k) \, = \, \dens (\cL)\, \widehat{g} (-k^{\star}) \ts ,
\]  
where $L^{\circledast}=\pi(\cL^{\ast})$, with $\cL^\ast$ the 
annihilator $($or dual lattice$\ts\ts )$ of $\cL$,
is the corresponding Fourier module.  \qed
\end{theorem}

It follows that $\widehat{\ts\gamma^{}_1 \ts}$ is pure point, with
Fourier module $L^\circledast = \ZZ[\tau]/\mbox{\small $\sqrt{5}$}$
and amplitudes
\[
   A^{}_{\!\vL, p}(k) \, = 
     \, \dens (\cL)\, \widehat{h} (-k^{\star}) 
     \, = \,  \dens (\vL) \ee^{\pi\ii(1+\sigma)k^{\star}}
             \sinc\big(\pi(1-\sigma)k^{\star}\big) \prod_{\ell=1}^{\infty} 
             \bigl( p+q \ee^{2\pi\ii \sigma^\ell k^{\star}} \bigr),
\]   
hence $\widehat{\ts\gamma^{}_1 \ts} =  \sum_{k\in L^{\circledast}} 
I_{p}(k)\, \delta_k$ with $I_p (k) = \lvert A^{}_{\!\vL, p}(k) \rvert^2$,
so
\begin{equation} \label{eq:ran-fib-int}  
   I_{p}(k)\, = \,  \left( \myfrac{\tau}
    {\mbox{\small $\sqrt{5}$}}  \ts
    \sinc (\pi\tau k^{\star} ) \right)^2 
    \prod_{\ell =1}^{\infty} \big|  p + q  
     \ee^{2\pi\ii \sigma^\ell k^{\star}} \big|^2
     \, = \, I(k) \prod_{\ell=1}^{\infty} \big|  p +q  
     \ee^{2\pi\ii \sigma^\ell k^{\star}} \big|^2 .
\end{equation}
Here, $I(k)$ is the intensity function from the deterministic case of
Theorem~\ref{theo:a} for $u^{}_{a} = u^{}_{b} = 1$.  Summarising the
above derivations, we obtain the following result.  Due to the
mentioned compatibility of the measures on
$\YY_{\!\zeta^{}_{\mathrm{F}}}$ and on the subset
$Y_{\nts \zeta^{}_{\mathrm{F}}}$, we may formulate it right away for
the entire dynamical system
$(\YY_{\!\zeta^{}_{\mathrm{F}}}, \RR , \nu^{}_{\nts \mathsf{pf}})$.

\begin{theorem}\label{theo:c} 
  Fix some\/ $\cT \in \YY_{\!\zeta^{}_{\mathrm{F}}}$ and let\/
  $\vL =\vL_a \ts \dot{\cup}\ts \vL_b$ be the corresponding set of
  left endpoints of the tiles in\/ $\cT$.  Then, almost surely with
  respect to the ergodic patch frequency measure\/
  $\nu^{}_{\!\mathsf{pf}}$, the corresponding diffraction measure
  reads
\[
   \widehat{\gamma} \, = \,
   \EE \bigl( \widehat{\ts\gamma^{}_{\!\vL}\ts} \bigr) \, = \,
    \widehat{\gamma}^{}_{\mathsf{pp}}
   + \widehat{\gamma}^{}_{\mathsf{ac}} \, =
   \sum_{k\in\ZZ[\tau]/\sqrt 5} \!
    I_{p}(k) \, \delta_k \ts\ + \; \phi^{}_{p} \ts \lm \ts ,
\] 
where\/ $I_{p}(k)$ is given by Eq.~\eqref{eq:ran-fib-int} and\/
$\phi^{}_{p}$ is the Radon--Nikodym density of\/
$\widehat{\gamma}^{}_{\mathsf{ac}}$, as computed previously in
\cite[Prop.~6.18]{mo}. \qed
\end{theorem}

\begin{remark}
  The original formula for $\phi^{}_{p}$ in \cite{mo,mo2}, which also
  appears in \cite{gl1}, can be made more explicit. Setting
  $\psi (k) = 1 - \cos \bigl( 2\pi \ts \frac{k}{\tau}\bigr)$, one
  finds
\[
    \phi^{}_{p} (k) \; = \; \frac{2\ts p q \ts \tau}
    {\mbox{\small $\sqrt{5}$}} \, \sum_{n=2}^{\infty}
    \, \frac{ \psi (k)}{\tau^{n}}  \,
    \prod_{\ell=1}^{n-2} \big| p + q \ts 
    \ee^{-2 \pi \ii \tau^{\ell} k} \big|^{2} ,
\]
with the understanding that an empty product is $1$. It is worthwhile
to note that $\phi^{}_{p}$ vanishes for the limiting cases $p=1$ and
$q=1$, where $p\ts q = 0$. This means that the formula for
$\widehat{\gamma}$ in Theorem~\ref{theo:c} remains valid, with the
correct result, for \emph{all} values $p\in [0,1]$.

One can also extend the explicit formulas to the case of arbitrary
weights $u^{}_a,u^{}_b \in \CC$ for the two types of points. The
intensities of the Bragg peaks are then the special case $m=1$ of the
formula in Theorem~\ref{thm:RNMS} below, while the above formula for
$\phi^{}_{p}$ remains true, this time with
$\psi(k) = \frac{1}{2} \ts\ts \bigl| (1-\ee^{-2 \pi \ii k})\ts  u^{}_a -
(1-\ee^{-2 \pi \ii \tau k}) \ts u^{}_{b} \bigr|^2$.  \exend
\end{remark}

One should notice that it is not immediately obvious why the
supporting set of the Bragg peaks in the stochastic case coincides
with that of the deterministic one. However, this ultimately
follows from the fact that the deterministic substitutions
$\zeta^{}_{\mathrm{F}\nts , 1}$ and $\zeta^{}_{\mathrm{F}\nts , 0}$
from Section~\ref{sec:fibo} give rise to the same hull, and that the
latter can be described within the same CPS as used
there. This CPS then also accommodates the generating random Fibonacci
sets, as described above.  An independent explicit argument for the
Fourier--Bohr spectrum via exponential sums can be found in
\cite{spi}; see also \cite{Nicu-book} for a general treatment in the
setting of almost periodic pure point measures.

\begin{remark}
  As before, we may also consider tiles with the modified lengths
  according to Eq.~\eqref{eq:tile-length}. In this case, one obtains
\begin{align*}
  A_{\vL,p}' (k) \, & = \, A_{\vL}'(k)  \prod_{\ell=1}^{\infty} 
    \bigl(  p +q  \ee^{2\pi\ii \sigma^\ell (k^{\star}-\rho k)} \bigr)
    \quad \text{and} \\
    I_{p}^{\ts\prime} (k) \, & = \, I^{\ts\prime}(k) \prod_{\ell=1}^{\infty} 
    \big|  p +q  \ee^{2\pi\ii \sigma^\ell (k^{\star}-\rho k)} \big|^2
\end{align*}
as generalisations of the previous expressions, which are covered
for $\rho = 0$.

Also, the density function $\phi^{}_{p}$ can be calculated for the
case with the modified interval lengths. With the new length function
$L^{}_{\ell} = \tau^{\ell} + \rho \ts\ts \sigma^{\ell}$ for the level-$\ell$
inflation words, one finds
\[
    \phi^{}_{p} (k) \; = \; \myfrac{2\ts p q \ts \tau}
    {\mbox{\small $\sqrt{5}$}} \, \sum_{n=2}^{\infty}
    \, \frac{ \psi (k)}{\tau^{n}}  \,
    \prod_{\ell=1}^{n-2} \big| p + q \ts 
    \ee^{-2 \pi \ii L^{}_{\ell} k} \big|^{2} ,
\]
now with $\psi (k) = 1 - \cos \bigl( 2 \pi (\rho\ts \tau +
 \sigma) k \bigr)$.    \exend
\end{remark}

\section{Random noble means substitutions}\label{sec:ran-nob}

Let us begin this section with a more general approach to the
concept of a random substitution; see \cite{RS} for some general
properties and results.

\begin{definition}\label{def:ransubst}
  Let a finite alphabet
  $\cA = \{ a^{}_{1} , a^{}_{2}, \ldots , a^{}_{n} \}$ be fixed.
  Then, an endomorphism
  \mbox{$\varrho \! : \, \cA^* \xrightarrow{\quad} \cA^*$} is called a
  \emph{random substitution} if there are $k_1,\ldots,k_n\in\NN$ and
  probability vectors
\[
  \big\{\bs{p}^{}_i=(p^{}_{i 1},\ldots,p^{}_{i k_i}) \mid 
   \bs{p}^{}_i\in[0,1]^{k_i}
   \text{ and } \sum_{j=1}^{k_i} p^{}_{ij}=1 \ts , \, 
   1\leqslant  i\leqslant  n\big \}
\] 
  such that
\[
    \varrho : \; a^{}_i\mapsto 
  \begin{cases} 
    w^{(i,1)}, & \text{with probability } p^{}_{i 1}, \\
    \quad \vdots  & \quad \quad \quad \quad \ \vdots \\
    w^{(i,k_i)}, & \text{with probability } p^{}_{i k_i},
\end{cases} 
\]  
for $1\leqslant i\leqslant n$, where each $w^{(i,j)}\in\cA^*$. 
Moreover, the average
\[
  M_{\varrho}\, := \Big(\sum_{q=1}^{k_j} p^{}_{jq} \,
   \card^{}_{a_i} \bigl(w^{(j,q)} \bigr) \Big)_{1\leqslant  i,j\leqslant  n}
   \in \mathrm{Mat}(n,\RR_{\geqslant 0})
\]
serves as the corresponding substitution matrix.
\end{definition}


\begin{remark}
  In principle, the integers $k_i$ may take the value $\infty$, but we
  do not consider such cases here.  As in the deterministic case, a
  random substitution $\varrho$ is \emph{primitive} if and only if
  $M_{\varrho}$ is a primitive matrix.  Various other notions can also be
  extended; compare \cite{RS}. 

  {}From a general point of view, each $\varrho(a_i)$ is a random
  variable, which means that our $M_\varrho$ is actually the
  \emph{expectation} of the substitution matrix, the latter also
  viewed as a random variable. We suppress such extensions, as we do
  not need them for our systems of compatible substitutions.  \exend
\end{remark}

The random Fibonacci substitution can be generalised as follows;
compare \cite{mo2} and references therein. Consider $\cA = \{ a,b \}$
as before, pick $m\in\NN$ and let $\bs{p}^{}_m=(p^{}_0,\ldots,p^{}_m)$
be a fixed probability vector. We shall usually assume that all $p_i>0$
unless specified otherwise. Define the deterministic substitutions
$\zeta^{}_{m,i}$ by $b \mapsto a \mapsto a^i b a^{m-i}$, for
$0\leqslant i \leqslant m$. These $m+1$ substitutions all define the
same hull, and share the substitution matrix
\[
      M^{(m)} \, = \, \begin{pmatrix} m & 1 \\ 1 & 0 \end{pmatrix}
\]
with PF eigenvalue
$\lambda^{}_{m} = \frac{1}{2} \bigl( m + \sqrt{m^2 + 4} \, \bigr)$,
which is a Pisot--Vijayaraghavan (PV) number and a unit.  Its 
algebraic conjugate,
$\lambda^{\prime}_{m} = \frac{1}{2} \bigl( m - \sqrt{m^2 + 4} \,
\bigr)$,
is the other eigenvalue of $M^{(m)}$. For the geometric realisation as
tilings with natural tile lengths, we choose $\lambda_m$ and $1$ for
the tiles (intervals) of type $a$ and $b$, respectively. One then
obtains the analogue of Theorem~\ref{theo:a}, this time with the
Fourier module
\begin{equation}\label{eq:FM}
   \cF_m \, = \, \myfrac{\ZZ[\lambda_m]}
   {\mbox{\small $\sqrt{m^2 + 4}$}} \ts ,
\end{equation}
which covers our previous case for $m=1$. 

Also, the description of $\vL =\vL_a \ts \dot{\cup}\ts \vL_b$ as model
sets is completely analogous. Here, the lattice is
$\cL_{m} = \{ (x,x^\star) \mid x\in\ZZ[\lambda_{m}] \}$, where
$x^\star = x'$ is algebraic conjugation in the quadratic field
$\QQ (\lambda_m) = \QQ \bigl(\nts \mbox{\small $\sqrt{m^2 + 4}$}
\ts\ts \bigr)$.
It turns out (see \cite{BG,mo} for background and details) that the
windows for $\zeta^{}_{m,j}$ with $0\leqslant j \leqslant m$ may be
chosen as
\[
    W^{(a)}_{m,j} \, = \, j \ts \tau^{}_m + [0,1)
    \quad \text{and} \quad
    W^{(b)}_{m,j} \, = \, j \ts \tau^{}_m + [\lambda_m', 0)
\]
with $\tau^{}_m = \frac{-1}{m} (\lambda_m'+1)$. The choice with the
half-open intervals is only relevant for $j=0$ and $j=m$, where the
fixed points are singular (as in the Fibonacci case); compare
\cite[Ex.~7.3]{BG} for a discussion of this point. Since the
hull defined by $\zeta^{}_{m,j}$ is independent of $j$, we get
the following result; see \cite{Q,BG} for background.

\begin{coro}\label{coro:detRNMS} 
  Fix\/ $m\in\NN$ and let\/ $\YY_{\! m}$ be the geometric hull of the 
  corresponding noble means tiling system,
  with prototiles of length\/ $\lambda_m$ for type\/ $a$ and\/ $1$ for
  type\/ $b$. Fix some\/ $\cT \in \YY_{\! m}$ and let\/
  $\vL =\vL_a \ts \dot{\cup}\ts \vL_b$ be the corresponding set of
  left endpoints of the tiles in\/ $\cT$.  Then, the weighted Dirac
  comb\/
  $\omega = u^{}_{a} \ts \delta^{}_{\!\vL_a} + \ts u^{}_{b} \ts
  \delta^{}_{\!\vL_b}$,
  with any fixed pair of weights\/ $u^{}_{a},u^{}_{b} \in \CC$, is
  pure point diffractive. Its autocorrelation is given by\/
\[
   \gamma \, = \sum_{z\in\vL-\vL}\, \sum_{\alpha,\beta \in \{a,b\}}
   \!\! \overline{u^{}_{\alpha}} \: \eta^{}_{\alpha \beta} (z)  \ts
   u^{}_{\beta} \; \delta_{z} \ts,
\]
where\/ $\eta^{}_{\alpha \beta} (z) = \dens \bigl( \vL^{}_{\alpha} \cap
(\vL^{}_{\beta} - z)\bigr)$, and the diffraction measure reads
$\widehat{\gamma} = \sum_{k\in \cF_m} I(k)\, \delta_k$, where\/ $\cF_m $
is the Fourier module from Eq.~\eqref{eq:FM}. As before, $I(k) = 
\bigl| u^{}_a A^{}_{\!\vL_a}(k) + u^{}_b A^{}_{\!\vL_b}(k) \bigr|^2$,
with the Fourier--Bohr coefficients from 
Theorem~\textnormal{\ref{theo:a}}.  In
particular, $\gamma$ and\/ $\vL-\vL$, as well as\/ $\widehat{\gamma}$
and\/ $I(k)$, are independent of\/ $\cT$.  \qed
\end{coro}

Now, given $m\in\NN$, the \emph{random noble mean substitution}
$\zeta^{}_m \! : \, \cA^* \xrightarrow{\quad} \cA^*$ is defined by
\begin{equation}\label{eq:RNMS}
\zeta^{}_m : \; 
\begin{cases}
      a\mapsto 
      \begin{cases}
      \zeta^{}_{m,0}(a), & \text{with probability } p^{}_0,\\
      \zeta^{}_{m,1}(a), &  \text{with probability } p^{}_1, \\
      \quad\ \vdots & \quad \quad \quad \quad \ \vdots \\
      \zeta^{}_{m,m}(a), & \text{with probability } p^{}_m,
      \end{cases} \\
      b\mapsto a,
\end{cases} 
\end{equation}
and the one-parameter family $\cR = \{\zeta^{}_m\}^{}_{m\in\NN}$ is
called the family of \emph{random noble means substitutions} (RNMS).
In particular, one has $\zeta^{}_{1} = \zeta^{}_{\mathrm{F}} $.

From here on, we can continue in close analogy to the Fibonacci
case. Eq.~\eqref{eq:concat-M} is now to be replaced by $m+1$
equations. Explicitly, one has
\[
   \cM^{}_{n}  \, = \sum_{r=0}^{j-1} \bigl( 
        \delta^{}_{r \ts \lambda^{n-1}_{m}} * 
      \cM^{(r)}_{n-1}\bigr) \, + \, 
      \delta^{}_{\nts j \ts \lambda^{n-1}_{m}} 
    * \cM^{}_{n-2} \,
    + \sum_{r=j}^{m-1} \bigl( 
    \delta^{}_{\lambda^{n-2}_{m} + r \ts \lambda^{n-1}_{m}}
     * \cM^{(r)}_{n-1}\bigr)
\]
with probability $p^{}_{j}$, for $0\leqslant j \leqslant m$.  Here,
empty sums are $0$ as usual, and
$\cM^{(0)}_{n-1} , \ldots , \cM^{(m-1)}_{n-1}$ are $m$ independent and
identically distributed copies of the random variable $\cM^{}_{n-1}$,
which is an important point to observe in comparison to the previous
case, $m=1$.

In this case, one gets
$\EE (\cM_n) = \sum_{x\in\smoplam(W_m )\cap \ts\RR_{\geqslant 0}} h^{(n)}
(x^\star ) \, \delta_x $, now with covering window
$W^{}_{\! m} = \bigl[\lambda_{m}' - 1, 1 - \lambda_{m}' \bigr]$ and
model set
$\oplam (W_{\nts m} )= \{ x \in \ZZ[\lambda_m] \mid x^\star \in
W_{\nts m} \}$.
As before, one has $h^{(n)} = h^{(n)}_{a} + h^{(n)}_{b} \! $, and with the
analogous definition of the measures $\mu^{(n)}_\alpha$ one arrives
again at a set of recursion relations. They can be used to establish
the existence of the limiting measures $\mu^{}_\alpha$, which then
satisfy the rescaling relations
\begin{align}\label{ms}
\begin{split}
     \mu^{}_a\, &=\, |\lambda_m'| \, \sum_{n=0}^m p^{}_n 
      \biggl(\, \sum_{j=0}^{n-1} (f^{}_j.\mu^{}_a)      
        \, + \sum_{j=n}^{m-1} (g^{}_j.\mu^{}_a) 
        \, + \, (f^{}_0.\mu^{}_b) \biggr) , \\
      \mu^{}_b\, &=\,  |\lambda_m'| \, \sum_{n=0}^m p^{}_n\, 
        (f^{}_n \nts .\mu^{}_a) \ts , 
\end{split}
\end{align} 
where $f^{}_j(x)=\lambda_m'(x+j)$ and $g^{}_j (x)=\lambda_m' (x+j)+1$,
while empty sums are $0$ by convention.  The solutions are again
absolutely continuous measures. They can be represented by
$\mu^{}_a=h^{}_a\lm$ and $\mu^{}_b=h^{}_b\lm$ with
\[
  h^{}_a(x) \, = \, \big(1_{[0,1)} *\mu\big)(x) 
   \quad \text{ and } \quad 
  h^{}_b(x) \, = \, \big(1_{[\lambda_m' , 0)} *\mu\big)(x) \ts ,
\] 
where we now have
\[
    \mu \, = \, \Conv_{\ell=1}^{\infty} \biggl( \,
    \sum_{n=0}^{m} p^{}_n \, \delta^{}_{n(\lambda_m')^\ell } \biggr) .
\] 
When $\bs{p}$ is strictly positive, the $L^{1}$-functions $h^{}_a$ and
$h^{}_b$ are again represented by continuous functions with compact
support (by an analogous argument), hence Theorem~\ref{theo} can be
applied.  This gives the following result.

\begin{theorem}\label{thm:RNMS}
  Let\/ $m\in\NN$ be fixed and consider the random noble means
  substitution from Eq.~\eqref{eq:RNMS}. Let\/ $\YY_{\!\zeta}$ be the
  geometric tiling hull, with intervals of length\/ $\lambda_m$ and\/
  $1$ as prototiles, and consider the dynamical system\/
  $(\YY^{}_{\!\zeta},\RR, \nu^{}_{\!\mathsf{pf}})$, with the ergodic
  patch frequency measure\/ $\nu^{}_{\! \mathsf{pf}}$.

  Then, for\/ $\nu^{}_{\!\mathsf{pf}}$-almost every element\/
  $\cT \in\YY_{\!\zeta}$, with\/ $\vL = \vL^{}_{a} \ts \dot{\cup}
  \ts \vL^{}_{b}$ denoting the left endpoints of\/ $\cT\!$,
  the diffraction measure of the weighted Dirac comb\/
  $\omega = u^{}_a \,\delta^{}_{\! \vL_a} + u^{}_b \,
  \delta^{}_{\! \vL_b}$ is of the form
\[
    \widehat{\ts\gamma^{}_{\!\vL}\ts} \, = \, 
    \widehat{\gamma}^{}_{\mathsf{pp}} +  \widehat{\gamma}^{}_{\mathsf{ac}} 
    \, = \sum_{k\in \cF_m} I^{}_{\bs{p}}(k) \, \delta^{}_{k}
    \, + \, \phi^{}_{\bs{p}} \ts \lm \ts ,
\]
with the Fourier module\/ $\cF_m=\ZZ[\lambda_m]/\sqrt{m^2 + 4}$. 
The Bragg peak intensities are given by
\[
  I^{}_{\bs{p}}(k) \, = \, I (k)
   \prod_{\ell=1}^{\infty} \, \Big|  \sum_{n=0}^m p^{}_n  
    \ee^{2\pi\ii n \ts (\lambda_m')^{\ell} k^{\star}} \Big|^2,
\]
where\/ $I (k) = \bigl|  \dens (\vL^{}_m) \,
\sinc\big(\pi(1-\lambda_m')k^{\star}\big) \bigr|^{2} $ with\/
$\dens (\vL^{}_{m}) = \frac{1-\lambda^{\prime}_{m}}{\sqrt{m^2 + 4}} $
is the `deterministic' part. Finally, the Radon--Nikodym density of 
the absolutely continuous part reads
\[
   \phi^{}_{\bs{p}} (k) \, = \, \frac{\lambda_m}{\sqrt{m^2 + 4 \ts}}
   \sum_{\ell=2}^{\infty} \frac{\psi^{(\ell)}_{\bs{p}} (k)}
     {\lambda^{\ell}_{m}} \ts ,
\]
where\/ $\psi^{(\ell)}_{\bs{p}}$ are uniformly bounded, continuous
functions on\/ $\RR$.  
\end{theorem}

\begin{proof}[Sketch of proof]
  As indicated, Theorem~\ref{theo} covers the case that $\bs{p}$ is
  strictly positive. It can then be shown that any of the remaining
  limiting cases where some of the $p^{}_i$ vanish is still covered by
  the same formula. Consequently, the result holds for all probability
  vectors and recovers the model set case with its pure point
  diffraction in the deterministic limits, where we have $p^{}_i=1$
  for a single index, in which case $\phi^{}_{\bs{p}}$ vanishes.
\end{proof}

Modified tile lengths can also be considered. Since the resulting
changes are structurally similar to those encountered in the
special case $m=1$, we omit further details, some of which will
be presented in \cite{diss-T}.

\section{Deterministic period doubling chain}\label{sec:pd}

In the next section, we are going to investigate a locally randomised
version of the \emph{period doubling} substitution. Therefore, let us
first recall what is known about the deterministic substitution, which
has constant length and is defined by
\[
  \rho^{}_{\mathrm{pd}}  : \quad
  a \mapsto ab \ts , \quad b \mapsto aa \ts ;
\]  
see \cite[Sec.~4.5.1, Ex.~7.4 and Sec.~9.4.4]{BG} for background. In
analogy to Theorem~\ref{theo:a}, one obtains the following result;
compare \cite[Ch.~V]{Q} and \cite[Sec.~9.4.4]{BG}.

\begin{theorem}\label{theo:b} 
  Let\/ $\YY_{\!\mathrm{pd}}$ be the geometric hull of the period
  doubling tiling system, with two distinct prototiles of length\/
  $1$, and\/
  $\YY_{\! 0} = \{ y \in \YY_{\! \mathrm{pd}} \mid 0 \in y \}$ its
  discrete counterpart. Then, the topological dynamical systems\/
  $(\YY_{\! 0},\ZZ)$ and\/ $(\YY_{\!\mathrm{pd}},\RR)$ are strictly
  ergodic, both with pure point dynamical spectrum.

  Now, fix some\/ $\cT \in \YY_{\!\mathrm{pd}}$ and let\/
  $\vL = \vL_a \ts \dot{\cup}\ts \vL_b$ be the corresponding set of
  left endpoints of the tiles in\/ $\cT$.  Then, the weighted Dirac
  comb\/ $\omega = u^{}_{a} \ts \delta^{}_{\!\vL_a} + u^{}_{b} \ts
  \delta^{}_{\!\vL_b}$
  with any fixed pair of weights\/ $u^{}_{a}, u^{}_{b} \in \CC$ is
  pure point diffractive. Its autocorrelation\/ $\gamma$ can be
  expressed as in Theorem~\emph{\ref{theo:a}}, while the diffraction
  measure reads
\[
    \widehat{\gamma} \, = \sum_{k\in\ZZ\left[\frac{1}{2}\right]}
     \! \! I(k)\, \delta_k \ts ,
\] 
with
$I(k) = \bigl| u^{}_{a} A^{}_{\!\vL_a}(k) + u^{}_{b} A^{}_{\!\vL_b}(k)
\bigr|^2$
and Fourier--Bohr coefficients defined as in
Theorem~\emph{\ref{theo:a}}.  In particular, $\gamma$ and\/ $\vL-\vL$,
as well as\/ $\widehat{\gamma}$ and\/ $I(k)$, are independent of\/
$\cT$, while the Fourier--Bohr coefficients do depend on the chosen
element, but converge uniformly.  \qed
\end{theorem}

Again, the elements of $\YY_{\!\mathrm{pd}}$ can be understood as
(translates of) regular model sets. For this purpose, choose
$H=\ZZ_2$, the set of $2$-adic integers, as locally compact Abelian
group to obtain a CPS $(\RR,\ZZ_2,\cL)$ with lattice
\[
   \cL  \, = \, \{(x,\iota(x)) \mid  x\in \ZZ\} 
   \, \subset \, \RR\times \ZZ_2 \ts ,
\]   
where $\iota \! : \, \ZZ \hookrightarrow \ZZ_2$ is the canonical
embedding, which is also the $\star$-map in this case. In particular,
one can describe the fixed point under $\rho^{2}_{\mathrm{pd}}$ with
seed $a|a$ as a regular model set in this way; see \cite{bm} as well
as \cite[Ex.~7.4]{BG} for details.

It is well known that the diffraction measure
$\widehat{\gamma^{}_{\mathrm{pd}}}$ of the corresponding Dirac comb,
which is also the diffraction measure of the entire system
$\YY_{\!\mathrm{pd}}$, is pure point. For generic choices of the
weights $u^{}_{a}$ and $u^{}_{b}$, the set of Bragg peak positions is
a group, namely
\begin{equation}\label{eq:pb-module}
  L^{\circledast} \, = \, \ZZ\bigl[\tfrac{1}{2}\bigr]
  \, = \, \left\{ \tfrac{m}{2^r}\ts  \big| \ts
   (r=0,\, m\in\ZZ) \text{ or } (r\geqslant 1,\, m \text{ odd})
  \right\} \ts ,
\end{equation} 
which means that there are then no extinctions. Let us mention in
passing that $L^\circledast$ is also the dynamical spectrum for the
dynamical system under the continuous translation action of $\RR$,
while the restriction of $L^{\circledast}$ to the $1$-torus $\TT$,
here written as $[0,1)$ with addition modulo $1$, is the spectrum for
the discrete $\ZZ$-action by the shift.

In our parametrisation, with $k=\frac{m}{2^r}$, the Fourier--Bohr
coefficients (or amplitudes) of our particular Dirac comb are given by
\[
  A^{}_{\!\vL_a}(k) \, = \, \myfrac{2}{3}\, 
  \frac{(-1)^r}{2^r}\, \ee^{2\pi\ii k}
  \quad \text{ and }  \quad 
  A^{}_{\!\vL_b}(k) \, = \, \delta^{}_{r,0} - A^{}_{\!\vL_a}(k) \ts .
\]
Hence, the diffraction intensities for $k\in L^{\circledast}$ can be
calculated as
\[
  I(k) \, = \, \begin{cases}
  \frac{1}{9 \cdot 4^{r-1}} \, \lvert u^{}_{a}-u^{}_{b} 
     \rvert^2  , &  r \geqslant 1 ,  \\[1mm]
  \frac{1}{9}\, \lvert 2 u^{}_{a} + u^{}_{b} \rvert^2, & r=0 .
  \end{cases}
\]   
Let us note in passing that $u^{}_{a} = u^{}_{b}$ leads to $I (k) = 0$
for all $k = \frac{m}{2^{r}}$ with $m$ odd and $r \geqslant 1$ because
the Dirac comb then `degenerates' to $u^{}_{a} \delta^{}_{\ZZ}$,
wherefore $\widehat{\gamma}$ simply becomes
$\lvert u^{}_{a} \rvert^{2} \ts \delta^{}_{\ZZ}$ and is `blind' to the
aperiodicity that is present for $u^{}_{a} \ne u^{}_{b}$.

Now, let us consider the alternative substitution
\[
  \rho^{\ts\ts \prime}_{\mathrm{pd}} \! : \quad 
  a \mapsto ba \ts , \;  b \mapsto aa \ts ,
\]  
which is conjugate to $\rho^{}_{\mathrm{pd}}$ by an inner automorphism
of the free group with generators $a$ and $b$.  By
\cite[Prop.~4.6.]{BG}, the substitutions $\rho^{}_{\text{pd}}$ and
$\rho^{\ts\ts \prime}_{\text{pd}}$ define the same two-sided hull.
Due to the constant-length nature of these substitutions, the symbolic
and the geometric picture coincide canonically, for instance by
choosing (coloured) intervals of unit length as prototiles (as we did above). In
particular, we can then identify the hulls $\XX_{\mathrm{pd}}$ and
$\YY_{\! 0}$.  Moreover, the relation between the topological
dynamical systems $(\YY_{\! 0},\ZZ)$ and $(\YY_{\! \mathrm{pd}},\RR)$
is given by a simple suspension with constant roof function; see
\cite{BS} for background. We shall see more of this in
Section~\ref{sec:pd-eigen}.

\begin{remark}
  The analysis of the system under a change of the tile lengths is
  considerably more involved here in comparison to the Fibonacci
  case. This is due to the nature of $\ZZ_2$ as internal space.  In
  particular, it is no longer true that changing the tile length ratio
  leads to a deformed model set. This can also be seen as a
  consequence of topological obstructions identified in \cite{CS}.
  \exend
\end{remark}

Clearly, global mixtures of $\rho^{}_{\mathrm{pd}}$ and
$\rho^{\ts\ts \prime}_{\mathrm{pd}}$ do not lead to an extension of
the hull, and are thus compatible in this sense, as in our previous 
examples. Once again, this situation changes under \emph{local} 
mixtures, as we shall see next.

\section{Random period doubling chain}  \label{sec:ran-pd}

Now, fix $p\in [0,1]$, set $q=1-p$, and define the \emph{random
  period doubling substitution} by
\begin{equation}\label{eq:def-ran-pd}
  \rho  : \; \begin{cases}
      a\mapsto 
      \begin{cases}
      ab \ts , & \text{with probability } p,      \\
      ba \ts , & \text{with probability } q,
      \end{cases} \\ b\mapsto aa \ts , \end{cases} 
\end{equation}
which has substitution matrix
$M=\left(\begin{smallmatrix} 1 & 2 \\ 1 & 0 \end{smallmatrix}\right)$,
independently of $p$.  As before, the term `random' is only justified
for $p\in (0,1)$, while the limiting cases correspond to the
deterministic cases of the previous section.  Again, for $p\in (0,1)$,
we define the two-sided discrete stochastic hull $\XX_{\rho}$ as
\[ 
    \XX_{\rho} \, := \, \{w\in\cA^{\ZZ}\ |\ 
    \mathfrak{F}(\{w\})\subseteq  \cD_{\rho}\} \ts ,
\]   
with the notation from Definition~\ref{def:hull}. It is clear by
construction that $\XX_\rho$ contains the deterministic hull as
a proper subset. 

\subsection{Entropy}
For the stochastic hull, one has the following result.

\begin{lemma}\label{lem:pd-entropy}
  For\/ $p\in (0,1)$, the topological entropy of\/ $(\XX_{\rho}, \ZZ)$
  is\/ $s=\frac{2}{3}\ts \log(2) \approx 0.462$.
\end{lemma}

\begin{proof}
  As long as $0<p<1$, the dictionary of legal words is always the
  same, and the hull contains elements with dense orbits under the
  $\ZZ$-action of the shift. Any element of such an orbit contains all
  finite legal words, wherefore the topological entropy $s$ equals the
  patch counting entropy of such an element; compare
  \cite{HR}. Consequently, if $\cW_{n}$ is the set of legal words of
  length $n$, one has $s = \lim_{n\to\infty} \frac{1}{n} \log \ts \lvert 
  \cW_{n} \rvert$, where $\lvert .\rvert$ denotes the cardinality of
  a set.  Note that the sequence $\bigl(\lvert \cW_{n} 
  \rvert\bigr)_{n\in\NN}$ is subadditive, wherefore the limit exists 
  by Fekete's lemma \cite{Fek}.

  Consider the pedigree graph of successive exact substitution words
  that originate from $a$ as its seed (or level $0$).  Due to the
  constant-length nature of the random period doubling substitution,
  this graph has the property that the words on any given level
  (defined by graph distance from $a$) are distinct. In other words,
  the graph is a \emph{tree}, with root $a$. This tree contains words
  of length $2^r$ for any $r\in\NN_{0}$, but only a subset of
  $\cW_{2^r}$ on level $r$, for any $r \in \NN_{0}$.

  It is easy to check inductively that all exact substitution words of
  length $2^{r}$ in this tree contain
  $\frac{1}{3} \bigl(2^{r+1} + (-1)^{r} \bigr)$ letters $a$. Moreover,
  since each $a$ (multiplicatively) gives rise to two distinct words
  on the next level, the total number $\#^{}_{r}$ of exact
  substitution words of length $2^{r}$, by induction, is given by
\[
    \#^{}_{r} \, = \, 2^{(2^{r+2} - (-1)^{r} - 3)/6},
\]
with
$\lim_{r\to\infty} 2^{-r} \log (\#^{}_{r}) = \frac{2}{3}\ts \log(2)$,
so that this clearly is a lower bound for $s$.

Now, consider a legal word of length $2n$, with $n\geqslant 2$
say. This word must either emerge as the substitution of a legal word
of length $n$, or of one of length $n-1$, then completed with a prefix
and a suffix of one letter each. If $m_{n}$ denotes the maximal number
of $a$'s in any element of $\cW_{n}$, we thus have the estimate
\[
   \lvert \cW_{2n} \rvert \, \leqslant \,
   2^{m_{n}} \lvert \cW_{n} \rvert + 4 \cdot
   2^{m_{n-1}} \lvert \cW_{n-1} \rvert \, < \,
   5 \cdot 2^{m_{n}} \lvert \cW_{n} \rvert \ts ,
\]
which implies
$s \leqslant \log(2)\ts \liminf_{n\to\infty} \frac{m_{n}}{n}$ by
standard arguments.  Since $\liminf_{n\to\infty} \frac{m_{n}}{n}$ is
bounded from above by $\frac{2}{3}$, which is the frequency of $a$'s
according to standard Perron--Frobenius theory with the substitution
matrix $M$, our lower bound for $s$ is also its upper bound, and the
claim on the entropy follows.
\end{proof}

\begin{remark}
  Let us note that the last part of the proof of
  Lemma~\ref{lem:pd-entropy} actually also shows that
  $\liminf_{n\to\infty} \frac{m_{n}}{n} = \frac{2}{3}$ and hence
\[
     \lim_{n\to\infty} \frac{m_{n}}{n} \, = \, \myfrac{2}{3} \ts .
\]
This follows from the observation that
$\limsup_{n\to\infty} \frac{m_{n}}{n} = \frac{2}{3} + \eps$ with
$\eps > 0$ would imply the existence of a sequence of words
$\bigl(w_{n_i}\bigr)_{i\in\NN}$ with the frequency of $a$'s converging
to $\frac{2}{3} + \eps$. But then,
$\bigl(\rho^{}_{\mathrm{pd}} (w_{n_i})\bigr)_{i\in\NN}$ would define
another sequence of words with the frequencies of $a$'s converging to
$\frac{2}{3} - \frac{\eps}{2} < \frac{2}{3}$, which is impossible.

An analogous argument shows that the minimal number of $a$'s in the
legal words of length $n$ asymptotically grows like $\frac{2}{3} \ts
n$ as well, and not slower, so that \emph{each} element of our stochastic
hull $\XX_{\rho}$ has well-defined frequencies $\frac{2}{3}$ and
$\frac{1}{3}$ for the letters $a$ and $b$, respectively.
\exend
\end{remark}

\subsection{Diffraction}
The diffraction measure is given by the obvious modification of
Eq.~\eqref{eq:expsum} to this case, where we need to consider distinct
weights for the two types of points.  We want to proceed by the same
method as before. Therefore, using the corresponding concatenation
rule, we obtain
\begin{equation}\label{eq:concat-2}
  \cX^{}_n(k) \, = \, 
  \begin{cases}
  \cX^{}_{n-1}(k) + \ee^{-2\pi\ii k \cdot 2^{n-1}}\cX^{}_{n-2}(k) + 
  \ee^{-2\pi\ii k \cdot 3\cdot 2^{n-2}}\cX_{n-2}'(k),
              & \text{with prob. } p, \\
  \cX^{}_{n-2}(k) + \ee^{-2\pi\ii k \cdot 2^{n-2}}\cX_{n-2}'(k) + 
  \ee^{-2\pi\ii k \cdot 2^{n-1}}\cX^{}_{n-1}(k), 
              & \text{with prob. }  q,
  \end{cases}
\end{equation} 
together with $\cX^{}_0(k)=u^{}_{a}$ and 
\[
  \cX^{}_1(k) \, = \, 
  \begin{cases}
  u^{}_{a} + u^{}_{b} \ee^{-2\pi\ii k}, & \text{with probability } p, \\
  u^{}_{b} + u^{}_{a} \ee^{-2\pi\ii k}, & \text{with probability } q.
  \end{cases}
\] 
Note that, in Eq.~\eqref{eq:concat-2}, $\cX^{}_{n-2}$ and
$\cX^{\prime}_{n-2}$ are \emph{independent} random variables with the
same distribution, as in our previous RNMS case.
   
As before, $(\cX_n(k))_{n\in\NN}$ is a subsequence of
$(X_m (k))_{m\in\NN}$, and we obtain
\[
    \EE \bigl( \widehat{\ts \gamma^{}_{\!\vL} \ts } \bigr)
   \, =  \lim_{n\to\infty} 
    \myfrac{1}{2^n}\, \big| \EE(\cX_n) \big|^2 
    +\lim_{n\to\infty} \myfrac{1}{2^n} \operatorname{Var}(\cX_n) 
    \,  =: \,  \widehat{\ts\gamma^{}_{1} \ts} + 
    \widehat{\ts\gamma^{}_{2} \ts}.
\]  
We have 
\[
   \widehat{\ts\gamma^{}_{1} \ts} \, = \lim_{n\to\infty} 
  \myfrac{1}{2^n}\, \big| \widehat{\EE(\cM_n)} \big|^2                  
  \, =  \lim_{n\to\infty} \myfrac{1}{2^n}\, 
   \cF \big[\EE(\cM_n)*\widetilde{\EE(\cM_n)}\big]  ,        
\]   
where $\cM_n$ is the measure with $\widehat{\cM_n}=\cX_n$. By
construction, $\EE(\cM)= \lim_{n\to\infty} \EE(\cM_n)$ is a weighted
Dirac comb on $\NN_0$, where the weight at $x\in\NN_0$ is given by
\[
   u^{}_{a} \, \PP ( \text{type at } x \text{ is } a ) +
   u^{}_{b} \, \PP ( \text{type at } x \text{ is } b ) \ts .
\] 
In analogy to our treatment in Section~\ref{sec:ran-fibo}, one now finds
\[
     \widehat{\ts \gamma^{}_{1} \ts} \, = \,
     \cF \big[\EE(\cM)\circledast \widetilde{\EE(\cM)}\big] ,
\]
where we tacitly assume that the volume-averaging for $\circledast$ is 
taken with the appropriate weights for one-sided sequences. 
The underlying reason is that the positive pure point measure
$\bigl| \EE (\cM)|^{}_{[0,2^n)} \! - \EE (\cM_n) \bigr|$ gets uniformly
small on $\ZZ \cap [0,2^n)$ as $n\to\infty$; see \cite{diss-T} for an
explicit derivation of this fact, and the Appendix for a general approach
that gives the result we need from a weaker assumption.

\subsection{Averages and weight function}
Next, we observe that we get
\begin{equation}\label{dirac-pd}
  \EE(\cM) \, = \, u^{}_{b} \ts \delta^{}_{\NN_0} + (u^{}_{a}-u^{}_{b})
   \sum_{x\in\NN_0} \! a_x\, \delta_x
\end{equation}
with $a_x = \PP ( \text{type at } x \text{ is } a )$. Thus, we can
restrict our attention to the case $u^{}_{a}=1$ and $u^{}_{b}=0$,
which reduces Eq.~\eqref{dirac-pd} to
$\EE(\cM) = \sum_{x\in\NN_0} a_x\, \delta_x$.  It is not difficult to
employ the random substitution to derive the following recursive
structure of the probabilities $a^{}_{x}$.

\begin{fact}\label{fact:a-prob}
  For any\/ $n\in\NN_{0}$, one has
\[
   a^{}_{2 n} \, = \, 1 - q \ts a^{}_{n}
   \quad \text{and} \quad
   a^{}_{2 n+1} \, = \, 1 - p \ts\ts a^{}_{n} \ts .
\]
In particular, this gives $a^{}_{0} = \frac{1}{1+ \ts q}$. \qed
\end{fact}

Next, we aim at defining a function
$h \! : \, \ZZ_{2} \xrightarrow{\quad} [0,1]$ in analogy to our
previous approach.  We begin by setting $h (n) = a^{}_{n}$ for any
$n\in\NN_0$, where we canonically identify $\NN_0$ with its image
$\iota (\NN_0)$ in the $2$-adic integers. To extend $h$, we will use a
uniform continuity argument, for which we first need an intermediate
result.

\begin{lemma} \label{lem:u-cont}
For arbitrary\/ $m,j,k \in\NN_0$, one has
\[
\begin{split}
  h(m\ts 2^j) \, & = \,  h(0) + \bigl( h(m) - h(0)
   \bigr)\ts (-q)^j \quad \text{and} \\
  h(m\ts 2^j+k) \, & = \, h(k) + x_j \ts , \quad \text{with }
    \, |x_j| \, \leqslant \,  (\max\{p,q\})^j .
\end{split}
\]  
\end{lemma}

\begin{proof} 
The first identity follows from a simple inductive calculation
(in $j$) with the recursion from Fact~\ref{fact:a-prob} for
$h (m\ts 2^j) = a^{}_{m\ts 2^j}$.

The second property can be shown by induction in $k$.  For $k=0$, and
any $m,j\in\NN_0$, the claim follows from the first identity, because
$\lvert h(m) - h(0) \rvert$ is bounded by $1$. Now, let the assertion
be true (for arbitrary $m,j\in\NN_0$) for $1,\ldots,k$, where we first
look at the case that $k=2r$ is even. Then, with $r\leqslant k$, one
gets
\[
\begin{split}
   h (m\ts 2^j+k+1)\,
   & = \, h \bigl( 2 (m\ts 2^{j-1}+r)+1 \bigr) 
    \, = \, 1 - p\, h (m\ts 2^{j-1} + r )  \\
   & = \, 1-p \ts \bigl( h (r)+x_{j-1} \bigr ) 
    \, = \, 1 - p \, h (r) - p\, x_{j-1} 
    \, = \, h (k+1) - x_{j} \ts ,
\end{split}
\]
where
$\lvert x_j \rvert \leqslant \max\{p,q\} \ts \lvert x_{j-1} \rvert$.
The case $k$ odd can be handled analogously.
\end{proof}

As before, in order to apply Theorem~\ref{theo}, the next aim is to
extend $h$ to a continuous function on all of $\ZZ_2$ such that
$a_n = h (n^{\star}) = h (\iota(n))$. We will see that it is enough to
consider the dense subset $\NN_0\subset\ZZ_2$, where the denseness of
$\NN_0$ is a consequence of that of $\ZZ$ together with the identity
$1+2+4+\ldots = -1$ in $\ZZ_2$. For $n\in\NN_0$, we have
$h (n^{\star}) = h (n)$ because the $\star$-map is the identity on
$\ZZ$, due to our canonical identification of $\ZZ$ with
$\iota (\ZZ)$.

\begin{lemma}\label{lem:h-cont}
  The function\/ $h$ can be extended to a uniformly continuous
  function on\/ $\ZZ_2$.
\end{lemma}

\begin{proof}
  Let $\eps>0$ be fixed, choose $j\in\NN$ such that
  $(\max\{p,q\})^j<\eps$ and set $\delta:=2^{-j}$. Then, we have
  $|x-y|^{}_2\leqslant \delta$ with $x,y\in\NN_0$ if and only if $x-y$
  is divisible by $2^j$, so $y=x + m \ts 2^j$ for some $m\in\NN_0$.
  Hence, by Lemma~\ref{lem:u-cont}, we obtain
\[
   |f(x)-f(y)| \, = \, |f(x)-f(x+m\ts 2^j)| \, = \,
   |x_j | \, \leqslant \, (\max\{p,q\})^j \, < \, \eps \ts .
\]   
As $\NN_0 \subset \ZZ_2$ is dense and $h$ is uniformly continuous on
$\NN_0$, we know that $h$ can be extended to a uniformly continuous
function on $ \ZZ_2 $; see \cite[(3.15.6)]{Dieu}. By slight abuse of 
notation, this extension is still called $h$.
\end{proof}

\begin{remark}
  While the argument we used to prove the continuity of $h$ in this
  section looks different from our previous argument, the basic idea
  is the same.  Indeed, consider a CPS $(G,H,\cL)$,
  and denote by $L = \pi(\cL) \subset G$ the projection of the
  lattice.  Then, the image $L^\star$ of $L$ under the star mapping 
  is dense in $H$.

  In the previous sections, we showed that the function $h$ is
  uniformly continuous by constructing it as an infinite convolution
  and showing the continuity of the convolution, which was made easy
  by the simple structure of the group $H=\RR$. This approach would
  also work here, but the computations are more involved due to the
  $2$-adic structure of $H$. What we did instead was to pull back $h$
  through the star mapping to the function $g(x)=h(x^\star)$, and show
  that $g$ is uniformly continuous in the induced topology of the
  embedding $\star \! : \, L \hookrightarrow H$. Then, the $2$-adic
  structure of $H$ makes the induced topology easy to work with, while
  the induced topology would be harder to tackle in the CPS of the
  previous sections.  \exend
\end{remark}

\subsection{Diffraction: Pure point part}
With this preparation, we define the two-sided comb
$\omega^{}_{h} = \sum_{x\in\ZZ} \, h(x) \, \delta_x$, and we then have
$\widehat{\ts \gamma^{}_{1}\ts} = \left(\omega^{}_{h} \nts\nts
  \circledast \widetilde{\omega^{}_{h}}\right)\!\widehat{\phantom{I}}
= \widehat{\omega^{}_{h * \widetilde{h}}}$
in complete analogy to our previous cases.  This is once again a pure
point measure, by another application of Theorem~\ref{theo}.\smallskip

Let us assume for a moment that $\widehat{\ts\gamma^{}_2 \ts}$ is a
continuous measure (which we will prove below). In this case, we can
apply a variant of \cite[Thm.~3.2]{Hof} and \cite[Thm.~5 and
Cor.~5]{Daniel} to obtain
\begin{equation}  \label{intensities}
    \EE \bigl(\widehat{\ts\gamma^{}_{\!\vL}\ts}\bigr) (\{ k \}) \, = 
    \lim_{n\to\infty}\myfrac{1}{4^n}\big| \EE(\cX_n(k))\big|^2.
\end{equation}  
From Eq.~\eqref{eq:concat-2}, setting
$E_n:=\EE(\cX_n)$, we infer that
\begin{equation} \label{E-rec}
    E_n \, = \, (p+q\ee^{-2^n\pi\ii k})\ts E_{n-1} + 
     (q+ q\ee^{-2^{n-1}\pi\ii k} + \, p \ee^{-2^n\pi\ii k}
          + \, p \ee^{-3\cdot 2^{n-1}\pi\ii k})\ts E_{n-2}
\end{equation}   
together with $E^{}_{0}(k)=1$ and $E^{}_{1}(k)=p+q\ee^{-2\pi\ii k}$. In
particular, $E^{}_{1} (\ell ) = 1$ for all $\ell \in\ZZ$.  Recall
that every $k\in L^{\circledast} = \ZZ\bigl[ \frac{1}{2}\bigr]$ can be
written in the form $k=\frac{m}{2^r}$. If $n\geqslant r+2$, one has
\[
    E_n \, = \, E_{n-1} + 2 \ts E_{n-2} \ts .
\]  
With the initial conditions $E_r$ and $E_{r+1}$, this recurrence
relation has the unique solution
\[
   E_n \, = \, \myfrac{1}{3} \left( 2^{n-r} \ts
  (E_r+E_{r+1}) + (-1)^{n-r}\ts (2E_r-E_{r+1}) \right).
\]  
Combining this with Eq.~\eqref{intensities}, we obtain
\[
   \EE \bigl(\widehat{\ts\gamma^{}_{\!\vL} \ts}\bigr)
   \left(\left\{\tfrac{m}{2^r}\right\}\right) \, = \,
    \tfrac{1}{9\cdot 4^{r}}
        \left| E_r \left(\tfrac{m}{2^r}\right)
    +E_{r+1}\left(\tfrac{m}{2^r}\right)\right|^2.
\]
Applying Eq.~\eqref{E-rec}, it is not difficult to see that
$E_{r+1}\left(\frac{m}{2^r}\right) =2 \ts
E_r\left(\frac{m}{2^r}\right)$. Hence, we have
\[
   \EE \bigl(\widehat{\ts\gamma^{}_{\!\vL} \ts}\bigr)
   \left(\left\{\tfrac{m}{2^r}\right\}\right) \, = \,
    \tfrac{1}{9\cdot 4^{r-1}}
        \left| E_r \left(\tfrac{m}{2^r}\right)\right|^2.
\]
Moreover, if we apply Eq.~\eqref{E-rec} inductively, we get
\[
\begin{split}
  E_r\left(\tfrac{m}{2^r}\right) \,
     &= \, \left(E_{r-1}\left(\tfrac{m}{2^r}\right)-
       \big(1+\ee^{-2^{r-1}\pi\ii\frac{m}{2^r}}\! \big)E_{r-2}
      \left(\tfrac{m}{2^r}\right)\right) 
          \ts  \left(-q-p\ee^{-2^r\pi\ii \frac{m}{2^r}}\right) \\[2mm]
     &= \, \left( E_1\left(\tfrac{m}{2^r}\right)-
     \big(1+\ee^{-2\pi\ii \frac{m}{2^r}}\big)
       E_0\left(\tfrac{m}{2^r}\right) \right)
       \ts \prod_{\ell=2}^r \left( -q-p
         \ee^{-2^\ell \pi\ii \frac{m}{2^r}}\right)\\[-1mm]
     &= \, \prod_{\ell =1}^r \left( -q-p
        \ee^{-2^\ell \pi\ii \frac{m}{2^r}}\right). 
\end{split}
\]
Finally, we arrive at
\begin{equation}\label{eq:ran-pd-int}
   \EE \bigl(\widehat{\ts\gamma^{}_{\!\vL} \ts}\bigr)
   \left(\left\{\tfrac{m}{2^r}\right\}\right)  
   \, = \,  \frac{1}{9\cdot 4^{r-1}}
        \prod_{\ell=1}^r \left| q + p
         \ee^{-2^\ell \pi\ii \frac{m}{2^r}} \right|^2 
   \, = \, \widehat{\gamma^{}_{\! \vL,\text{det}}}  
       \left(\left\{\tfrac{m}{2^r}\right\}\right)  
       \ts  \prod_{\ell=1}^r\left|  q + p
          \ee^{-2^\ell \pi\ii \frac{m}{2^r}} 
        \right|^2. 
\end{equation}

\begin{remark}
  As in the case of the random Fibonacci substitution, we find that
  the Bragg peak intensities in the stochastic situation are given by
  the deterministic ones multiplied by some function that depends on
  the probabilities $p$ and $q$. For example, one obtains
\[
   \EE \bigl(\widehat{\ts\gamma^{}_{\! \vL}\ts}\bigr)
   \left(\bigl\{\tfrac{1}{2}\bigr\}\right)
   \, = \, (p-q)^2 \, \widehat{\gamma^{}_{\! \vL,\text{det}}}
   \left( \bigl\{\tfrac{1}{2}\bigr\}\right) ,
\]
while $ \widehat{\ts\gamma^{}_{\! \vL}\ts} (\{ 0 \}) =
\widehat{\gamma^{}_{\! \vL,\text{det}}} (\{ 0 \}) = \frac{4}{9}$,
independently of $p$ and for every $\vL \in \YY_{\!\mathrm{pd}}$.
\exend
\end{remark}

\subsection{Diffraction: Continuous part}
Let us now focus on $\widehat{\ts\gamma^{}_{2}\ts }$. Following the
same idea as in \cite[Sec.~6.2.3]{mo}, we obtain a recurrence relation
for $V_n:=\operatorname{Var}(\cX_n)$,
\[
   V_n \, = \, V_{n-1} + 2V_{n-1} + 
     2\ts p\ts q \ts\ts \psi_n, \quad \text{for } n\geqslant 2 \ts ,
\]
with $V_0\equiv0$ and $V_1(k)=2pq\big(1-\cos(2\pi k)\big)$, where the
functions $\psi_n$ are defined by
\[
  \psi_n(k) \, := \, \myfrac{1}{2} \ts \bigl| (1-\ee^{-2\pi\ii
  2^{n-1} k })E_{n-1}-(1+\ee^{-2\pi\ii 2^{n-2} k } -\ee^{-2\pi\ii
  2^{n-1} k }-\ee^{2\pi\ii 3\cdot 2^{n-2} k })E_{n-2} \bigr|^2.
\]
Via induction, the functions $\psi_n$ can be expressed more explicitly
as follows.

\begin{lemma}\label{lem:psi}
We have
\[
   \psi_n(k) \, = \, \big( 1-\cos(2^n\pi k)\big) 
   \prod_{\ell=1}^{n-1} \left| \ts q+p\ee^{-2^\ell \pi\ii k}\right|^2.
\]
In particular, we obtain\/ $\psi_n(k)\leqslant 2$ for all\/
$n\geqslant 2$.  \qed
\end{lemma}

This leads to the following observation.

\begin{prop} \label{prop:dens-ran-pd}
  The measure\/ $\widehat{\ts\gamma^{}_{2}\ts}$ is absolutely continuous
  with respect to\/ $\lm$. Its Radon--Nikodym density is given by the
  continuous function
\[
  \phi^{}_{p} (k) 
     \, = \, \myfrac{1}{3} V_1 + \frac{4p\ts q}{3} 
          \sum_{j=2}^{\infty} 2^{-j} \psi_j(k)     \\
     \, = \, \frac{4p\ts q}{3}\sum_{n=1}^{\infty} 
      \frac{1-\cos(2^n\pi k)}{2^n} 
            \prod_{j=1}^{n-1} \left|\ts q+p\ee^{-2^j\pi\ii k} \right|^2.
\]
\end{prop}

\begin{proof}
  Define $\alpha_n=\frac{1}{3}\, (2^n-(-1)^n)$. One can easily show
  via induction that
\[
    V_n \, = \, \alpha_n V_1 + 2\ts p \ts q 
      \sum_{j=2}^n \alpha_{n+1-j} \, \psi_j,
\]
which implies
\[
  \frac{V_n}{2^n} \, = \, \myfrac{1}{3}\, \frac{2^n-(-1)^n}{2^n}\,
   V_1 +  2\ts p\ts q \sum_{j=2}^n \myfrac{1}{3}\, 
  \frac{2^{n+1-j}-(-1)^{n+1-j}}{2^n}\, \psi_j.
\]
Now, it is not difficult to see that $\frac{V_n(k)}{2^n}$ converges
uniformly to $\phi^{}_{p}(k)$. The second equality then follows from
Lemma~\ref{lem:psi}.
\end{proof}

If we collect our findings, we obtain the following result.

\begin{theorem}\label{theo:d} 
  Fix some\/ $\cT \in \YY_{\!\rho}$ and let\/
  $\vL =\vL_a \ts \dot{\cup}\ts \vL_b$ be the corresponding set of
  left endpoints of the tiles in\/ $\cT$.  Then, almost surely with
  respect to the patch frequency measure\/ $\nu^{}_{\! \mathsf{pf}}$
  of the random period doubling substitution, the corresponding
  diffraction measure reads
\[
    \widehat{\gamma} \, = \,
    \widehat{\gamma}^{}_{\mathsf{pp}} +  
    \widehat{\gamma}^{}_{\mathsf{ac}} \, =
    \sum_{k\in\ZZ[1/2]} \!
    I_{p}(k) \, \delta_k \ts\ + \; \phi^{}_{p}\ts \lm \ts ,
\] 
where\/ $I_{p} (k)$ is given by Eq.~\eqref{eq:ran-pd-int} and\/
$\phi^{}_{p}$ is the Radon--Nikodym density of\/
$\widehat{\gamma}^{}_{\mathsf{ac}}$, as stated in
Proposition~\textnormal{\ref{prop:dens-ran-pd}}. \qed
\end{theorem}

\section{Eigenfunctions and Kronecker factor}\label{sec:decompose}

It is a common feature of the above examples that their pure point
part could be determined in closed form, and deeply resembled the
formulas known from weighted model sets, though the systems themselves
do certainly not belong to this class. Moreover, in view of positive
topological entropy and the structure of the hulls, one cannot expect
all eigenfunctions to be continuous. Consequently, the maximal
equicontinuous factor (MEF) will not be the right tool to proceed.

Instead, one has to identify the Kronecker factor, which emerges as
the maximal pure point factor under a measurable map, where we are
allowed to work up to a null set of the hull. In fact, this is where
the covering model set will enter, and the well-defined MEF of this
model set will be the Kronecker factor of our compatible random
inflation systems.

\subsection{Covers and eigenfunctions}

To explain what happens, we will first discuss the random Fibonacci
inflation from Section~\ref{sec:ran-fibo}, which means that we
consider the dynamical system $(\YY, \RR, \nu)$ with
$\YY = \YY^{}_{\!\zeta^{}_{\mathrm{F}}}$ and
$\nu = \nu^{}_{\nts\mathsf{pf}}$. Recall that $a$ and $b$ stand for
intervals of length $\tau$ and $1$, respectively, both with a
reference point on their left end. Consider all possible infinite
one-sided tilings (to the right of the origin, that is) that emerge as
a realisation of an infinite inflation process from a single tile, $a$
say, with its reference point at $0$. In this situation, if we code a
point by a pair $(\alpha, x)$ with $\alpha \in \{ a, b\}$, one
inflation step, on the level of individual points, means
\[
     (b,x) \mapsto (a, \tau x)  \quad \text{and} \quad
     (a,x) \mapsto \begin{cases}
     \{ (b, \tau x), (a, \tau x + 1) \} , & \text{with prob.\ $p$} , \\
     \{ (a, \tau x), (b, \tau x + \tau) \} , & \text{with prob.\ $q$} . 
     \end{cases}
\]
Let us assume $pq\ne 0$, which excludes the two deterministic cases,
and let $\vL_\alpha$ be the union of all type-$\alpha$ positions of
\emph{all} realisations, and consider the sets
$W_{\alpha} := \overline{\vL_{\alpha}}$ in internal space. It is clear
that these sets must satisfy the system of equations given by
\begin{equation}\label{eq:IFS-1}
\begin{split}
   W_a \, & = \, \sigma W_a \cup \sigma W_b \cup
       (\sigma W_a +1) \ts , \\[1mm]
   W_b \, & = \,  \sigma W_a \cup (\sigma W_a + \sigma)  \ts ,
\end{split}
\end{equation}
where $\sigma = \tau'$ as before. Since $\lvert \sigma \rvert < 1$,
\eqref{eq:IFS-1} defines a contractive iterated function system on
$\cK \times \cK$, where $\cK$ is the set of compact subsets of $\RR$
equipped with the Hausdorff metric; see \cite{BM-self,Wicks} for
background.

\begin{fact}
  The compact sets\/ $W_a = [-1,\tau]$ and\/ $W_b = [-\tau,1/\tau]$
  are the unique solution to Eq.~\eqref{eq:IFS-1} within\/
  $\cK \times \cK$, and one has\/ $W = W_a \cup W_b = [-\tau, \tau]$.
\end{fact}

\begin{proof}
  Since $\cK \times \cK$ with the Hausdorff metric is a complete
  metric space, and the iteration defined by the right-hand side of
  \eqref{eq:IFS-1} is a contraction in this space, with contraction
  constant $\lvert \sigma \rvert$, Banach's contraction principle
  guarantees a unique fixed point, which then is the solution of
  \eqref{eq:IFS-1}.  One can now check by a simple calculation that
  the intervals $W_a$ and $W_b$ as stated solve this equation.
\end{proof}

One immediate consequence is that, for each possible realisation, the
points of type $\alpha$ are a subset of the regular model set
$\oplam (W_{\alpha})$ in the CPS of the Fibonacci chain from
Eq.~\eqref{eq:CPS}. The question now is to what extent these windows
are determined by a single realisation.  For an answer, we employ the
so-called `chaos game' \cite[Sec.~4.2]{Palmer} and Elton's ergodic
theorem \cite{Elton}; see also \cite[Thm.~10]{Barnsley} or
\cite[p.~127]{Palmer}.

Consider the single{\ts}-point iteration in internal space, as defined by
$p(0)= (a,0)$ together with $p(n+1) = \Theta (p(n))$ for
$n\geqslant 0$, where $\Theta$ is a random mapping in internal space,
defined by
\[
     (b,y) \mapsto (a, \sigma y)
     \quad \text{and} \quad
     (a, y) \mapsto \begin{cases}
     (a, \sigma y + 1) , & \text{with prob.\ $p^{}_{1} $}  , \\
     (a, \sigma y) , & \text{with prob.\ $p^{}_{2} $}  , \\
     (b, \sigma y) , & \text{with prob.\ $p^{}_{3} $}  , \\
     (b, \sigma y + \sigma) , & \text{with prob.\ $p^{}_{4} $}  , 
     \end{cases}
\]
where $p_i >0$ and $\sum_i p_i =1$. Now, Elton's theorem asserts that,
almost surely, the corresponding (infinite) random point sequences lie
dense in the attractor of the IFS, as long as all $p_i>0$.

In direct space, each such sequence is an (exponentially thin) subset
of a possible realisation, and the previous argument shows that
already this thin subset, almost surely, has a dense lift into the two
windows. This establishes the following result.

\begin{prop}\label{prop:fix-window}
  Almost every realisation of the one-sided random Fibonacci inflation
  tiling completely determines the windows of the covering
  two-component model set, in the sense that the lift of the positions
  of type\/ $\alpha$ via the\/ $\star$-map lies dense in the compact
  set\/ $W_{\alpha}$. In particular, the lift of all left endpoints
  together is a dense subset of\/ $W=[-\tau,\tau]$.  \qed
\end{prop}

The corresponding result applies to one-sided tilings that extend
to the left, for instance when starting from $(b,-1)$ or from
$(a,-\tau)$. By intersecting two sets of full measure, one obtains
the following consequence.

\begin{coro}\label{coro:two}  
  Almost every realisation of the two-sided random Fibonacci inflation
  tiling that emerges from one of the central seeds\/ $a|a$, $a|b$,
  $b|a$ or\/ $b|b$ completely determines the window of the covering
  model set, as in Proposition~\textnormal{\ref{prop:fix-window}}. \qed
\end{coro}

Let $Y_0$ denote the set (or fibre) of all two-sided realisations
according to Corollary~\ref{coro:two}.  Of course, there are
realisations in $Y_0$ that do \emph{not} fix the window, such as the
ones that give perfect Fibonacci chains, but all such cases together
are only a null set. Here, as mentioned earlier, the relevant measure
on $Y_0$ is the one induced by $\nu$ on it via filtration, and agrees
with the one defined by the inflation process according to
Figure~\ref{fig:infl}; see \cite{PG}.  This also means that some
elements of $Y_0$ are thinnings not just of one model set, but of
many, and this applies analogously to all translates $Y_t = t +
Y_0$.
This is the origin of the discontinuity of non-trivial eigenfunctions,
as we analyse next.

Recall that $0 \ne f\in L^{2} (\YY, \nu)$ is called an
\emph{eigenfunction} of $(\YY, \RR, \nu)$ if there exists a $k\in\RR$
such that
\begin{equation}\label{eq:def-ef}
       f( t + y) \, = \, \ee^{2 \pi \ii k t} f(y)
\end{equation}
holds for all $t\in\RR$ and $\nu$-a.e.\ $y\in\YY$. Moreover, $f$ is
called \emph{continuous} if there is a continuous function on $\YY$
such that \eqref{eq:def-ef} holds for all $y\in\YY$. In this case, $k$
is called a \emph{topological eigenvalue} (in additive notation).

\begin{prop}\label{prop:fibo-cont}
  The topological point spectrum of\/ $(\YY,\RR,\nu)$ is trivial,
  which is to say that the only continuous eigenfunction is the
  constant one.
\end{prop}    

\begin{proof}
  Let $f$ be a continuous eigenfunction. Since $(\YY, \RR, \nu)$ is
  ergodic, $\lvert f \rvert$ is a constant, which we may choose to be
  $1$.  Let $y\in Y_0$ be fixed, and set $c=f(y)$. Due to the
  structure of the fibre $Y_0$, the tiling $y$ is of the form
  $y^{}_{\mathrm{L}} | \ts y^{}_{\mathrm{R}}$, which is to say that it
  consists of two infinite half-tilings that are glued together at
  $0$. Let $y^{\ts \prime}\in Y_0$ be any other element, which is then
  of the form
  $y^{\ts\prime}_{\mathrm{L}} | \ts y^{\ts\prime}_{\mathrm{R}}$.  Now,
  $Y_0$ clearly also contains the element
  $y^{}_{\mathrm{L}} | \ts y^{\ts\prime}_{\mathrm{R}}$, and one has
\[
    f (y^{\ts\prime}) \, = \, f ( y^{\ts\prime}_{\mathrm{L}} | \ts
    y^{\ts\prime}_{\mathrm{R}} ) \, = \, f(y^{}_{\mathrm{L}} |
    \ts y^{\ts\prime}_{\mathrm{R}}) \, = \, f( y^{}_{\mathrm{L}} |
    \ts y^{}_{\mathrm{R}} ) \, = \, f(y) \, = \, c \ts ,
\]
because continuous eigenfunctions cannot distinguish between two
right-asymptotic or between two left-asymptotic
elements. Consequently, $f$ is constant on $Y_0$.

Now, consider the two inflation tilings that correspond to the two
fixed points of the square of the Fibonacci inflation
$b \mapsto a \mapsto ab$, with seeds $a | a $ and $b | a$, called
$y^{}_{1}$ and $y^{}_{2}$.  By construction, both are elements of
$Y_0$. At the same time, the left endpoints (of the tiles of types $a$
and $b$) are regular model sets, given by
\[
    \oplam^{\! (a)} \big( [\tau-2,\tau-1) \bigr) \ts , \; 
    \oplam^{\! (b)} \bigl( [-1,\tau-2)\bigr)  \quad \text{and} \quad
    \oplam^{\! (a)} \bigl( (\tau-2,\tau-1] \bigr) \ts , \;
    \oplam^{\! (b)} \bigl( (-1,\tau-2] \bigr) ,
\]
respectively.  A comparison with the windows for $\vL_a$ and $\vL_b$
now shows that also $\tau + y^{}_{1}$ and $1 + y^{}_{2}$ are in $Y_0$,
wherefore we may conclude that $f( y^{}_{1}) = f(\tau + y^{}_{1}) =
\ee^{2 \pi \ii \tau k} f (y^{}_{1})$ and $f( y^{}_{2}) = f (1 + y^{}_{2}) =
\ee^{2 \pi \ii k} f ( y^{}_{2})$. This implies $k=0$ and $f$ is thus the 
constant function as claimed.
\end{proof}

This shows why we cannot work with the MEF of $(\YY,\RR,\nu)$,
which is trivial, but need to consider its Kronecker factor instead.

\subsection{Torus parametrisation and Kronecker factor}

It will be instrumental to employ the MEF of another dynamical system
as follows. Let $\vL = \oplam (W)$ with $W=[-\tau,\tau]$ be the
covering model set from above, in the CPS $(\RR, H, \cL)$ from
\eqref{eq:CPS}, where we use $H=\RR$ to explicitly distinguish direct
and internal space in our following arguments. It follows from the
standard theory of model sets via dynamical systems \cite{BLM} that
$\vL$ defines a strictly ergodic dynamical system that is a.e.\ one to
one over its MEF. The latter is $\AAA = (\RR \nts \times \! H)/\cL$
together with the induced translation action of $\RR$ on it.  Here,
$\AAA$ is a $2$-torus, and a translation by $t\in \RR$ is represented
as a translation by $(t,0)$ modulo $\cL$ on $\AAA$. Moreover, we also
have the classic \emph{torus parametrisation} at hand, where we assume
that $\vL$, which is a singular element, is the union of all elements
in the fibre over $(0,0)\in\AAA$.

The connection now works as follows. The fibre $Y_0$ is linked to
$\vL$ itself and hence mapped to $(0,0)$. Since a.e.\ element in the
fibre determines the window of $\vL$ uniquely by
Corollary~\ref{coro:two}, we can unambiguously map these elements to
$(0,0)$. To extend this to a mapping from $\nu$-a.e.\ element of $\YY$
to $\AAA$, we first select a generic element $y^{}_{0} \in
Y^{}_{0}$. Now, for any $y\in\YY$, there is a sequence 
$(t_n)^{}_{n\in\NN}$ of translations such that
\begin{equation}\label{eq:limit}
     y \, = \lim_{n\to\infty} (t^{}_{n} + y^{}_{0}) \ts .
\end{equation} 
It clearly suffices to consider the transversal of $\YY$, which is to
say that we may assume
$y\in\YY^{}_{0} := \{ u \in \YY \mid 0 \in u \}$ without loss of
generality. The advantage is that we now always have
$y\subset \ZZ[\tau]$, so all $t_n$ in \eqref{eq:limit} lie in
$\ZZ[\tau]$ as well, and the $\star$-map is well defined. Also, the
convergence then simply means that we may choose $t_n$ such that
$y\cap [-n,n] = (t^{}_n + y^{}_{0}) \cap [-n,n]$ holds, because our
point sets have finite local complexity.

\begin{lemma}\label{lem:torus}
  If\/ $(r,s)$ is a cluster point of\/ $(t_n , 0)^{}_{n\in\NN}$ in\/
  $\AAA$, with the translations\/ $t_n$ from Eq.~\eqref{eq:limit}, we
  have\/ $\, -r + y \ts \subseteq \nts \oplam (-s + W)$.
\end{lemma}

\begin{proof}
  Let $U$ and $V$ be open, relatively compact neighbourhoods of
  $0\in\RR$ and $0\in H$, respectively, and assume $V=-V$. Then, our
  assumption implies that there is a subsequence $(n_j)^{}_{j\in\NN}$ 
  of integers such that
\begin{equation}\label{eq:inset}
   (t^{}_{n_j}, 0) \in (r,s) + U\!\times\nts V \nts + \cL
\end{equation}
holds for all sufficiently large $j$, say $j>N$. For any such $j$, we
have $n_j > j$ and thus
\[
   y \cap [-j,j] \, = \, (t^{}_{n_j} + y^{}_{0} ) \cap [-j,j]
   \, \subseteq \, t^{}_{n_j} + \ts y^{}_{0} \, \subseteq \,
   t^{}_{n_j} + \oplam (W) \ts .
\]
By \eqref{eq:inset}, we have
$(t^{}_{n_j} , 0) = (r,s) + (u,v) + (x, x^\star)$ for some $u\in U$,
$v\in V$ and $(x,x^\star )\in\cL$, hence $t^{}_{n_j} = r + u + x$ and
$s + v + x^\star = 0$. Consequently, we have
$y \cap [-j,j] \subseteq r + u + x + \! \oplam (W)$, where
$x + \! \oplam (W) = \oplam(x^\star + W)$ because $(x,x^\star ) \in \cL$.
This implies
\[
   y \cap [-j,j] \, \subseteq \, r + u + \! \oplam (-s - v + W)
   \, \subseteq \, r + U \nts + \nts \nts \oplam (-s + W \! - V) \, = \,
   r + U \nts + \nts \nts \oplam (-s + W \! + V) \ts ,
\]
which holds for all $j > N$ and thus implies
\[
    y \, =  \bigcup_{j>N} y \cap [-j,j] \, \subseteq \,
    r + U \nts + \nts \nts \oplam (-s +W \! +V) \ts .
\]
Since this holds for any open neighbourhood $U$ of $0$, and
since $\oplam (-s+W \! +V)$ is a Delone set due to the relative
compactness of $V$, we get
\[
    \bigcap_{ 0 \in U \,\text{open} } \!\!
    U + \bigl( r + \oplam(-s + W \! + V) \bigr) \; = \;
    r + \oplam (-s + W \! + V)
\]
so that $y \subseteq r + \oplam (-s +W \! +V)$ and hence also
\[
    y \, \subseteq \bigcap_{\substack{0\in V = -V \\ 
    V \, \text{open}} } \!\!
    r + \! \oplam (-s + W \! + V) \, \supseteq \,
    r + \! \oplam (-s + W) \ts .
\]
Now, our claim follows if we show that the last inclusion actually
is an equality. 

To do so, we may assume $r=0$ without loss of generality.  Let
$x\in L \setminus \oplam (-s +W)$, where $L = \pi (\cL)$ from the CPS
\eqref{eq:CPS}, so $x^\star \notin - s +W$. Then, there is an open
neighbourhood $V$ of $0\in H$ with $V=-V$ such that
$(x^\star + V) \cap (-s +W) = \varnothing$, which implies that
$x^\star \notin -s+W \! +V$ and thus $x \notin \oplam (-s+W \!
+V)$. Consequently, $y \in r + \! \oplam (-s+W)$ as claimed.
\end{proof}

Next, in order to define a proper mapping from (a subset of) $\YY$
to the torus, we need to get rid of the subsequences from the previous
lemma.

\begin{coro}\label{coro:torus-2}
  If\/ $y \in \YY$ is generic, there exists a unique\/
  $(r,s)  \in \AAA$ such that
\[  -r+y  \, \subseteq \oplam (-s +W) \ts . \]
\end{coro}

\begin{proof}
  Since $\AAA$ is compact, any sequence $(t_n , 0)^{}_{n\in\NN}$ in 
  $\AAA$ has at least one cluster point. Therefore, Lemma~\ref{lem:torus} 
  gives the existence, and it remains  to show uniqueness. 

  Let $-r_i + y \subseteq \! \oplam (-s_i +W)$ for $i\in\{ 1,2 \}$, hence
  also the inclusions $(-r_i + y)^\star \subseteq -s_i + W $ and thus
  $s_i + (-r_i + y)^\star \subseteq W$, which means that the sets
  $-r_i + y$ are translates of elements in our special fibre,
  $Y_0$. When $y$ is generic (in the measure-theoretic sense), the
  window is uniquely determined, which is to say that
\[
     \overline{s_i + (-r_i +y)^\star} \, = \, 
     s_i + \overline{(-r_i + y)^\star} \, = \, W ,
     \quad \text{for  $\, i \in \{ 1,2 \}$}  \ts .
\]
But this implies
\[
\begin{split}
   -s^{}_2 + W   & = \, \overline{(-r^{}_{2} + y)^\star} \, = \,
   \overline{(-r^{}_{2} + r^{}_{1} - r^{}_{1} + y)^\star} \\
   & = \,   (r^{}_{1} - r^{}_{2})^\star + \overline{(-r^{}_{1} + y)^\star}
   \, = \, (r^{}_{1} - r^{}_{2})^\star - s^{}_{1} + W .
\end{split}
\]
Since $v + W = W$ is only possible for $v=0$, we conclude that
$s^{}_{1} - s^{}_{2} = (r^{}_{1} - r^{}_{2})^\star$, which means
nothing but $(r^{}_{1} - r^{}_{2} , s^{}_{1} - s^{}_{2}) \in \cL$
and our claim follows.
\end{proof}

At this point, we can define
\[
    \YY' \, := \, \{ y \in \YY :  \text{there is a unique }
    (r,s) \in \AAA \text{ with } -r + y
    \subseteq \oplam (-s + W) \} \ts ,
\]
and we then have a well-defined mapping $\psi \! : \,
\YY' \xrightarrow{\quad} \AAA$. By Corollary~\ref{coro:torus-2},
$\YY'$ contains all generic elements and thus has full measure.

\begin{prop}
   The mapping\/ $\psi \! : \, \YY' \xrightarrow{\quad} \AAA$
   is continuous.
\end{prop}

\begin{proof}
Since $\RR$ is metrisable, the same property holds for $\YY$ and
$\AAA$, and we may work with sequences. Let $y^{}_{n} \in \YY'$
with $n\in\NN$
be chosen so that $y^{}_{n} \!\xrightarrow{\quad} y$ in $\YY'$ 
as $n\to\infty$. We then need to show that $\psi (y^{}_{n} ) 
\xrightarrow{\quad} \psi (y)$.

Let $\psi (y^{}_{n}) = (r^{}_{n}, s^{}_{n})$ and $\psi (y) = (r,s)$.
Since $\AAA$ is compact, it suffices to show that any cluster
point of $(r_n, s_n)$ equals $(r,s)$ modulo $\cL$. Let $(r', s')$ be
a cluster point of the sequence, so $(r^{}_{k_n}, s^{}_{k_n}) 
\xrightarrow{\quad} (r', s')$ modulo $\cL$ for a suitable 
subsequence $(k_n)^{}_{n\in\NN}$.

By Lemma~\ref{lem:torus}, we have $ y^{}_{k_n} \subseteq 
 r^{}_{k_n} +  \oplam ( -s ^{}_{k_n} + W) $.
Hence, for all open neighbourhoods $U$ of $0$ in $G$
and $V$ of $0$ in $H$, there is some $N_1$ so that
$(r^{}_{k_n}, s^{}_{k_n}) + \cL \in (r' + U, s' + V) + \cL$
holds for all $n > N_1$, and thus
\[
    y^{}_{k_n} \, \subseteq \, r' + U + \oplam (-s' - V + W) \ts .
\]
Now, let $A>0$. Then, as $y^{}_{k_n}\! \xrightarrow{\quad} y$, 
there is some $N_2$ such that
\[
    y \cap [-A, A] \, = \, y^{}_{k_n} \cap [-A,A] \, \subseteq \,
    y^{}_{k_n} \, \subseteq \, r' + U + \oplam (-s' - V + W)
\]
holds for all $n > N_2$. Since this applies to all $A>0$, we get
$y \subseteq r' + U + \oplam (-s' - V + W)$.

Now, since we have this for all open neighbourhoods $U,V$ 
 as specified, we have
\[
   y \, \subseteq \, \bigcap_{U,V} r' + U + \oplam (-s' -V +W)
   \, = \, r' + \oplam (-s' +W)
\]
which shows that $-r' + y \subseteq \oplam (-s' +W)$. By the
uniqueness of the parameter $(r,s)$ attached to $y \in \YY'$, we
get $(r', s') = (r,s)$ modulo $\cL$ as desired.
\end{proof}

At this point, for each character
$\chi \!  : \, \AAA \xrightarrow{\quad} \CC$, the mapping
$\chi \circ \psi$ defines an eigenfunction of $(\YY, \RR, \nu)$ that
is continuous on $\YY'$. This complements the statement of
Proposition~\ref{prop:fibo-cont}.  We can now formulate the main
result of this section as follows.

\begin{theorem}\label{thm:MEF}
  The Kronecker factor of the dynamical system\/ $(\YY, \RR, \nu)$ can
  be identified with the MEF of the dynamical system obtained from the
  covering model set. It is explicitly given by\/
  $\AAA = (\RR \nts \times \! H) /\cL$ within the CPS~\eqref{eq:CPS},
  with\/ $H=\RR$.
\end{theorem}

\begin{proof}
  The mapping $\psi \! : \, \YY' \xrightarrow{\quad} \AAA$ from above
  is the measure-theoretic factor map onto $\AAA$. The maximality of
  this factor is a consequence of Theorem~\ref{theo:c}, as the dual
  group of $\AAA$ is precisely the Fourier module of the pure point
  spectrum, which is tantamount to saying that the mappings
  $\chi\circ \psi$ on $\YY'$ account for \emph{all} eigenfunctions of
  our system.
\end{proof}

Another approach, via a different view on the projection method,
was recently suggested by Keller and Richard \cite{KR}. In this
setting, as detailed in \cite[Def.~2.2]{Keller}, the notion of an `almost' 
MEF appears naturally, and is called a \emph{maximal equicontinuous generic
factor}, or MEGF for short. It is defined as
\[
     \YY'' \, := \, \{  y \in \YY : \overline{\RR + y} = \YY \} \ts ,
\]
which is closely related to the set $\YY'$ defined earlier. Now, due
to Corollary~\ref{coro:torus-2}, we can apply \cite[Thm.~2.4
(iii)]{Keller} to
$\psi \! : \, \YY' \cap \YY'' \xrightarrow{\quad} \AAA$ to derive that
$\AAA$ is a factor of the MEGF, and that $\psi$ extends to a
continuous mapping $\psi \! : \, \YY'' \xrightarrow{\quad}
\AAA$.
Since we already know that $\AAA$ is the Kronecker factor, being a
factor of $\YY''$ means that it is metrically isomorphic to the MEGF.

In fact, one can see that $\YY'' \subseteq \YY'$ in this situation. If
$y \in \YY''$ and $y^{}_{0} \in Y_0$, we can find a sequence
$(t_n)^{}_{n\in\NN}$ so that $t_n + y \xrightarrow{\quad}
y^{}_{0}$ as $n\to\infty$. Since $y^{}_{0}$ determines the
window $W$ by Proposition~\ref{prop:fix-window}, a 
computation similar to the ones at the end of 
Lemma~\ref{lem:torus} and  Corollary~\ref{coro:torus-2} 
shows that $y \in \YY'$. Putting the pieces together gives the 
following alternative view to our constructive approach.

\begin{coro}
  In the setting and notation of 
  Theorem~\textnormal{\ref{thm:MEF}}, the 
  following assertions hold.
\begin{enumerate}\itemsep=2pt
 \item  $\AAA$ is the MEGF
      of\/ $\YY$, and\/ $\psi \! : \, \YY' \xrightarrow{\quad} \AAA$
      is continuous.
  \item  The eigenfunctions of\/ $(\YY, \RR, \nu)$ are
      continuous on\/ $\YY'$. 
  \item $(\YY, \RR, \nu)$ is not weakly mixing.
 \qed
\end{enumerate}  
\end{coro}

\subsection{Interpretation via disintegration}

Now, consider the regular model set $\vL = \oplam (W)$ and the
dynamical system obtained as the orbit closure under the
$\RR$-action. This is a uniquely ergodic system with pure point
spectrum, and it is a.e.\ $1:1$ over its MEF, which is a $2$-torus in
our case at hand. This one also acts as the Kronecker factor for our
system $(\YY,\RR,\nu)$, where the map is only defined for $\nu$-almost
every element of $\YY$ by first identifying the unique covering model
set and then projecting down to the MEF.

The MEF of the covering model set is the compact Abelian group $\AAA$,
which is the \mbox{$2$-torus} equipped with Lebesgue measure as its
Haar measure. Here, the translation action is represented by a group
addition with dense range, as mentioned earlier. Now, over every
$a\in\AAA$, we have a fibre $Y_a \subset \YY$ together with a
probability measure $\nu^{}_{a}$ on it. For $a=0$, this is just our
special fibre $Y_0$ from above. These fibre measures are compatible
with the (normalised) Haar measure on $\AAA$ as needed for a
disintegration formula. For our dynamical system $(\YY, \RR, \nu)$ and
any $f \in L^1 (\YY, \nu)$, we then have
\begin{equation}\label{eq:disintegration}
     \EE (f) \, =  \int_{\YY} f (y) \dd \nu (y) 
     \, = \int_{\AAA} \int_{Y_{a}}
     f (y) \dd \nu^{}_{a} (y)  \dd a  \, = \int_{\AAA}
     \EE (f | Y_a) \dd a \ts ,
\end{equation}
in line with the general theory; see \cite[Ch.~5.4]{Fuerst}.

Analogous expressions hold for measure-valued quantities.  When the
inner integral is translation invariant on $\AAA$, as is the case for
the autocorrelation and diffraction measures encountered earlier, the
expectation can be obtained from the conditional expectation over one
fibre, say $Y_0$, which is precisely the approach taken in the
previous sections.

\begin{remark}
  The analogous procedure also works, step by step, for the random
  noble means inflations from Section~\ref{sec:ran-nob}. The
  topological point spectrum is once again trivial. Here, the proof
  uses the existence of elements $y^{}_{1}$ and $y^{}_{2}$ in the
  special fibre $Y_0$ such that $1 + y^{}_{1}$ as well as
  $\lambda^{}_{m} + y^{}_{2}$ also lie in $Y_0$, with the same
  conclusion as before; see \cite{diss-T} for details.  Also,
  non-trivial eigenfunctions are only discontinuous on a null set.
  
  Moreover, the entire structure with the covering model set and its
  dynamical system carries over, thus establishing the MEF of the
  model set as the Kronecker factor of $(\YY, \RR, \nu)$, and as its 
  MEGF. The disintegration then works the same way as in the 
  Fibonacci case.
  \exend
\end{remark}

\subsection{Random period doubling chain}\label{sec:pd-eigen}

Here, the situation is slightly different for two reasons. First, the
substitution \eqref{eq:def-ran-pd} is of constant length, which means
that we can identify the discrete and the tiling picture and work with
the $\ZZ$-action of the shift.  Second, the connection to a model set
requires a $2$-adic internal space, so that handling windows is more
complicated.

Let $\XX^{}_0$ denote the discrete hull, and $(\XX^{}_0, \ZZ, \nu)$
the corresponding dynamical system, with $\nu$ denoting the patch 
frequency measure, which is ergodic. As before, $\XX^{}_0$ contains 
a special fibre, denoted $X_0$, which contains all elements that are
realisations in the form of two level-$\infty$ supertiles (or
superwords) meeting at $0$. Here, the eigenfunction
equation takes the form
\[
      f (n + x) \, = \, \ee^{2 \pi \ii k n} f (x)
\]
for some $k \in \TT$, the dual group of $\ZZ$, and then all $n\in\ZZ$.
We represent $\TT$ as the half-open interval $[0,1)$ with addition
modulo $1$.

\begin{prop}\label{prop:cont-pd}
  The topological point spectrum of\/ $(\XX^{}_0, \ZZ, \nu)$ is
  trivial, which is to say that the only continuous eigenfunction is
  the constant one.
\end{prop}

\begin{proof}
  Let $f$ be a continuous eigenfunction. As in the proof of
  Proposition~\ref{prop:fibo-cont}, $\lvert f \rvert$ is continuous
  and invariant and hence constant (as $(\XX_0, \ZZ, \nu)$ is
  ergodic). Moreover, $f$ is again constant on the special fibre
  $X_0$.
  
  Next, observe that $\XX_0$ contains a periodic element, namely the
  one obtained by periodic repetition of the $3$-letter word $aab$.
  What is more, it is contained in the fibre $X_0$ in three different
  ways, as is apparent from
\[
   \begin{split}
   \cdots \underline{b\ts a} \, \underline{ab} \, \underline{aa} \, 
   \underline{b\ts a} \, \underline{ab} \ts & | \ts 
   \underline{aa} \, \underline{b\ts a} \, \underline{ab} \,
   \underline{aa} \, \underline{ba} \cdots \\
   \cdots \underline{ab} \, \underline{aa} \, \underline{ba} \,
   \underline{ab} \, \underline{aa} \ts & | \ts 
   \underline{b\ts a} \, \underline{ab} \, \underline{aa} \,
   \underline{b\ts a} \, \underline{ab} \ts \cdots \\
   \cdots \underline{aa} \, \underline{ba} \, \underline{ab} \, 
   \underline{aa} \, \underline{b\ts a} \ts & | \ts 
   \underline{ab} \, \underline{aa} \,  \underline{b\ts a} \, 
   \underline{ab} \, \underline{aa}  \cdots \ts .
   \end{split}
\]  
Since $f$ takes the same value on all three, which are translates
of one another, we get
\[  
       \ee^{2 \pi \ii k} \, = \, 1  \quad \text{with } \, k \in \TT ,
\]
which implies $k=0$. Thus, $f$ must be the constant eigenfunction 
as claimed.
\end{proof}

As we can already see from our diffraction analysis in
Section~\ref{sec:ran-pd}, the measure-theoretic point spectrum of
$(\XX_0, \ZZ, \nu)$ is given by
$\TT \cap \ZZ\bigl[ \frac{1}{2} \bigr]$. As before, $\XX^{}_0$
contains an open set of full measure, $\XX^{\prime}_{0}$ say, with the
property that all eigenfunctions are continuous on it. The
discontinuity is thus once again caused by a null set in the hull.

\begin{remark}
  Via a suspension with a constant roof function, the discrete
  dynamical system $(\XX_0, \ZZ, \nu)$ can be embedded into a flow,
  written as $(\XX, \RR,\nu^{}_{\RR})$ with $\nu^{}_{\RR}$ being the
  standard extension of $\nu$ to an invariant probability measure on
  $\XX$.  This system is also ergodic, and the topological point
  spectrum becomes $\ZZ$, while the measure-theoretic point spectrum
  is all of $\ZZ \bigl[ \frac{1}{2}\bigr]$. The additional continuous
  eigenfunctions in comparison to Proposition~\ref{prop:cont-pd}
  trivially emerge from the suspension. In terms of the approach via
  the Fourier--Bohr coefficients, this can be seen by adding a complex
  weight of the form $\ee^{2 \pi \ii n k}$ with a fixed $n\in\ZZ$,
  which results in a phase change for the continuous flow, but remains
  invisible for the discrete shift.  \exend
\end{remark}

{}From here, the remainder of the argument is similar to before.  We
get a covering two-component model set, and its MEF as the Kronecker
factor of $(\XX, \RR, \nu^{}_{\RR} )$. Moreover, there exists a continuous
mapping from the set $\XX'$ to the Kronecker factor, and $\XX'$
contains all elements with dense orbit. In particular, the Kronecker
factor is also the MEGF, and all eigenfunctions are continuous on
$\XX'$. As a consequence, we also have
the disintegration as in Eq.~\eqref{eq:disintegration}, which explains
the nice formulas we were able to obtain in Section~\ref{sec:ran-pd}.
At this point, we leave further details to the interested reader.

\section{Outlook}

The focus of this article was on compatible substitutions or
inflations that are ultimately related to a regular model set via an
implicit thinning process. This made the spectral structure fully
accessible. In general, the situation will be more complex, in
particular as far as the relation between the topological point
spectrum and the measure-theoretic one is concerned. Thus, an approach
in several steps seems most promising.

First, one could consider semi-compatible inflations that still share
the same substitution matrix, but do no longer define the same hull;
for one concrete example, where the complete determination of the
diffraction measure is still possible, we refer to \cite{PG}. Next,
one could relax the connection to model sets, and consider local
mixtures of substitutions with singular continuous spectrum. As long
as they still share the left PF eigenvector of the substitution
matrix, concrete results should still be possible because a consistent
geometric realisation exists under this condition.

Whether more general mixtures (such as one between the Fibonacci and
the \mbox{Thue{\ts}--Morse} substitution, which is sometimes called
`Fib-Morse') will lead to reasonable results is presently unclear, but
somewhat doubtful. Prior to such an attempt, a better understanding of
the general structure of random substitutions and their hulls is
needed, where we refer to \cite{RS} for some first systematic steps.

\section*{Appendix}

In several places in the main manuscript, we have used approximation
results for the autocorrelation of a translation bounded pure point
measure.  Our concrete justification for these steps was based on a
certain type of uniform convergence of the approximating measures on
finite regions of growing size (in fact, we have used a one-sided
version of such an approximation). Clearly, the concrete criteria were
sufficient, but certainly not necessary. In this appendix, we look at
this situation in a slightly more systematic way. We do this first for
an approximation by unbounded (but still translation bounded)
measures, and then for the case that one uses \emph{finite} measures
as approximations, which is the typical scenario in inflation-based
systems.

Below, we formulate our results for measures in $\RR^d$, and refer to
\cite{future} for a more general setting. Note that there is also a
complementary selection of results on the Fourier side of the coin,
that is, there are several approximation results for the diffraction
measure of a given translation bounded measure
$\omega \in \cM^{\infty} (\RR^d)$. Though this is both interesting and
relevant in its own right, we concentrate on the autocorrelation
measures here, and refer to \cite{BG,future} for results on their
Fourier transforms.

\subsection{Approximation by unbounded measures}

Consider a sequence $(\mu^{}_n )^{}_{n\in\NN}$ of translation bounded
measure such that $\mu^{}_{n} \xrightarrow{\,n\to\infty\,} \mu$ in the
vague topology. Assume that a van Hove sequence
$\cA = (A^{}_{m})^{}_{m\in\NN}$ is given such that the
autocorrelations
$\gamma^{}_{n} = \mu^{}_{n} \circledast \widetilde{\mu^{}_{n}}$ and
$\gamma = \mu \circledast \widetilde{\mu}$ all exist along $\cA$. When
is it true that also $\gamma = \lim_{n\to\infty} \gamma^{}_{n} $
holds?  That this must fail in general can easily be seen from an
example, such as
$\mu^{}_{n} = \delta^{}_{2\ZZ} + \delta^{}_{(2\ZZ +1) \cap [-n,n]}$.
Here, one has $\gamma = \mu = \delta^{}_{\ZZ}$, but
$\gamma^{}_{n} = \frac{1}{2}\ts \delta^{}_{2\ZZ}$, for all $n\in\NN$.
It is thus clear that one needs some other relation between $\mu$ and
$\mu^{}_{n}$. In particular, as we shall see, we do \emph{not} need
vague convergence, while convergence in a different topology is what
counts. This is common in the diffraction context, as nicely outlined
in \cite{Bob-tale}.\smallskip

Let a van Hove sequence $\cA$ be fixed, with all $A_m$ compact.  We
assume $A_m \subset A_{m+1}$, together with the usual condition that
$\vol (\partial^K \! A_m) = \scO \ts (\vol (A_m))$ as $m\to\infty$ for
any compact $K\subset \RR^d$, where $\partial^K \nts S$ is the
$K$-boundary of the set $S$; see \cite{BG} for details. There are
various possible generalisations of this setting, which we omit here.
We need $\cA$ for averages of various kinds, such as the
volume-averaged or \emph{Eberlein convolution} $f\circledast g$ of two
(locally integrable) functions $f$ and $g$, as (pointwise) given by
\begin{equation}\label{eq:fun-Eberlein}
       \bigl( f \circledast g \bigr) (x) \, := \lim_{m\to\infty}
       \frac{\bigl(f |^{}_{A_m} \! * g |^{}_{A_m} \bigr) (x)}
       {\vol (A_m)} \ts ,
\end{equation}
whenever this limit exists. Here, $f |^{}_{A_m}$ denotes the
restriction of $f$ to $A_m$ and $*$ the ordinary convolution of
functions. Operations of this type are needed when dealing with almost
periodic functions and related objects. Note that
$f \circledast g = 0$ if $f$ or $g$ has compact support.

Let $C_{\mathsf{u}} (\RR^d)$ denote the space of uniformly continuous
and bounded functions. The \emph{mean}  of
$ f \in C_{\mathsf{u}} (\RR^d)$ relative to $\cA$ is defined as
\[
     M (f) \, := \lim_{m\to\infty} \myfrac{1}{\vol (A_m)}
     \int_{A_m}  \! f (x) \dd x  \ts ,
\]
provided the limit exists. This is certainly the case for all weakly
almost periodic functions, but not for all
$ f \in C_{\mathsf{u}} (\RR^d)$. In contrast, one can define the
\emph{upper absolute mean} along $\cA$ as
\begin{equation}\label{eq:abs-mean}
     \overline{M} (f) \, := \, \limsup_{m\to\infty} \,
     \myfrac{1}{\vol (A_m)}
     \int_{A_m} \! \lvert f (x) \rvert \dd x \ts ,
\end{equation}
which clearly exists for all $ f \in C_{\mathsf{u}} (\RR^d)$ and
satisfies
$\overline{M} (f) = \overline{M} \bigl( \lvert f \rvert \bigr)$.  This
way, $\overline{M} (.)$ defines a semi-norm on
$C_{\mathsf{u}} (\RR^d)$. There are obvious variants of this
definition, but we only need this simple version below.  Here, a
sequence $(f^{}_{n})^{}_{n\in\NN}$ of functions from
$C_{\mathsf{u}} (\RR^d)$ is \emph{mean convergent} to
$g\in C_{\mathsf{u}} (\RR^d)$ if
$\overline{M} (f^{}_{n}\nts - g) \xrightarrow{\, n\to\infty\,} 0$. We
use $f^{}_{n} \rightsquigarrow g$ to denote this type of convergence.

Now, let $f,g,h \in C_{\mathsf{u}} (\RR^d)$ and assume that both
$f\circledast h$ and $g\circledast h$ exist relative to $\cA$, as
defined by Eq.~\eqref{eq:fun-Eberlein}. It is elementary to verify 
the estimate
\begin{equation}\label{eq:mean-est}
   \| f \circledast h - g \circledast h \|^{}_{\infty} \, 
   \leqslant \, \overline{M} (f-g) \, \| h \|^{}_{\infty} \ts ,
\end{equation}
which has the following important consequence.

\begin{fact}
  Let\/ $f^{}_{n}, g, h \in C_{\mathsf{u}} (\RR^d)$, where\/
  $n\in\NN$. Assume that\/ $f^{}_{n} \rightsquigarrow g$ as\/
  $n\to\infty$ and that\/ $g \circledast h$ as well as\/
  $f^{}_{n}\nts \circledast h$ exists for all\/ $n\in\NN$ with respect
  to a given van Hove sequence\/ $\cA$.  Then, one has\/
  $\,\lim_{n\to\infty} \| f^{}_{n} \nts \circledast h - g \circledast
  h \|^{}_{\infty} = 0$.  \qed
\end{fact}

The crucial observation here is that the convergence in mean for the
$f^{}_{n}$ implies a much stronger type of convergence after Eberlein
convolution. It is a known phenomenon that the Eberlein convolution
usually has nicer properties than the original functions.  Indeed, the
Eberlein convolution of two functions is weakly almost periodic, and
the Eberlein convolution of two weakly almost periodic functions
becomes uniformly (or Bohr) almost periodic.  \smallskip

Let $\cM^{\infty} (\RR^d)$ denote the space of translation bounded
Radon measures on $\RR^d$, which we primarily see as continuous linear
functionals over $C_{\mathsf{c}} (\RR^d)$, the space of continuous
functions with compact support, but we also identify Radon measures
with regular measures over the Borel $\sigma$-algebra by the general
Riesz--Markov representation theorem for this case \cite{Dieu}. Now,
given a sequence $(\mu^{}_{n})^{}_{n\in\NN}$ of measures from
$\cM^{\infty} (\RR^d)$, we need convergence in different
topologies. The standard one is \emph{vague convergence}, denoted by
$\mu^{}_{n} \rightarrow \mu$, which means
$\lim_{n\to\infty} \mu^{}_{n} (f) = \mu(f)$ for all
$f\in C_{\mathsf{c}} (\RR^d)$.  Since
$\mu (f) = \bigl(\mu * f_{\text{-}} \bigr) (0)$, where
$f_{\text{-}} (x) = f(-x)$, one can equivalently characterise vague
convergence via
$\lim_{n\to\infty} \bigl(\mu^{}_{n} \nts * g) (0) = \bigl( \mu *
g\bigr) (0)$
for all $g\in C_{\mathsf{c}} (\RR^d)$.  It is easy to see that this is
actually equivalent to the seemingly stronger
$\mu^{}_{n} \nts *g \xrightarrow{\, n\to\infty\,} \mu*g$ pointwise on
$\RR^d$, for all $g \in C_{\mathsf{c}} (\RR^d)$.

Next, we speak of \emph{norm convergence} to a measure
$\mu\in\cM^{\infty} (\RR^d)$, denoted by $\mu^{}_{n} \Rightarrow \mu$,
if
$\| \mu^{}_{n} \nts - \mu \|^{}_{K} \xrightarrow{\, n\to\infty\,} 0$
for some (fixed) compact set $\varnothing \ne K\subset \RR^d$ that is
the closure of its interior, where
$\| \mu \|^{}_{K} := \sup_{t\in\RR^d} \lvert \mu \rvert (t+K)$.  Note
that any $K$ with $\varnothing \ne K=\overline{K^{\circ}}$ defines the
same topology.  Next, we speak of \emph{convergence in the product
  topology}, denoted by
$\mu^{}_{n} \stackrel{\pi\,}{\rightarrow} \mu$, if one has
$\| (\mu^{}_{n} \nts - \mu) * g \|^{}_{\infty} \xrightarrow{\,
  n\to\infty\,} 0$
for all $g\in C_{\mathsf{c}} (\RR^d)$; and of \emph{mean convergence},
as before denoted by $\mu^{}_{n} \rightsquigarrow \mu$, if
$\mu^{}_{n} \! * g \rightsquigarrow \mu*g$ as $n\to\infty$ holds for
all $g\in C_{\mathsf{c}} (\RR^d)$.  The topology induced by mean
convergence can also be induced by the family of semi-norms given by
$\lvert \mu \rvert^{}_{g} := \overline{M} (\mu * g )$ with
$g\in C_{\mathsf{c}} (\RR^d)$. It is not a Hausdorff topology.

Via standard estimates, one can now verify the following relations. 

\begin{lemma}\label{lem:topologies}
  For translation bounded measures\/ $\mu^{}_{n}$, with\/ $n\in\NN$,
  and\/ $\mu$, one has the  implications
\[
     \bigl( \mu^{}_{n} \Rightarrow \ts \mu \bigr)
     \;\Longrightarrow \;
     \bigl( \mu^{}_{n} \stackrel{\pi\,}{\rightarrow} \ts \mu \bigr)
     \;\Longrightarrow \; \begin{cases}
     \bigl( \mu^{}_{n} \rightsquigarrow \ts \mu \bigr) \\
     \bigl( \mu^{}_{n} \rightarrow \ts \mu \bigr)   
     \end{cases}  
\]   
none of which is reversible in general. Moreover, considering mean
versus vague convergence, one has that neither of them implies the
other. \qed
\end{lemma}

Let us give a few examples to illustrate
Lemma~\ref{lem:topologies}. As $n\to\infty$, one has
$\delta^{}_{1/n} \stackrel{\pi\,}{\rightarrow} \delta^{}_0$, but no
convergence in the norm topology. Likewise, $\delta_{n} \rightarrow 0$
and $\delta_{n} \rightsquigarrow 0$, but no convergence in the product
topology. The measures
$\mu^{}_{n} = \delta^{}_{2 \ZZ} + \delta^{}_{(2\ZZ+1)\cap[-n,n]}$ from
above satisfy $\mu^{}_{n} \rightarrow \delta^{}_{\ZZ}$ and
$\mu^{}_{n} \rightsquigarrow \delta^{}_{2\ZZ}$. Note that the last
relation could change if we replace Eq.~\eqref{eq:abs-mean} by some of
the possible variants. In this sense, one has to be careful with the
possible notions of mean convergence \cite{future}.

\begin{remark}
  There are important further relations between the above topologies
  if one restricts them to suitable subclasses of translation bounded
  measures. Let us simply mention some of them, and refer to
  \cite{future} for proofs and further details, in the more general
  setting of locally compact Abelian groups. Here, let 
  $\vL \subset \RR^d$ be a uniformly discrete point set, and select 
  an open neighbourhood $U\subset \RR^d$ of $0$ such that, for any
  $x,y \in \vL$ with $x\ne y$, the sets $x+U$ and $y+U$ are
  disjoint. Further, fix a compact set $K\subset U$ with non-empty
  interior that also satisfies $K=-K$, which is clearly possible.
   
  Now, consider pure point measures
  $\mu , \mu^{}_{n} \in\cM^{\infty} (\RR^d)$, with $n\in\NN$, which
  are all supported in $\vL$. For any $f\in C_{\mathsf{c}} (\RR^d)$
  with $\supp (f) \subset U$, we then have
  $\lvert (\mu^{}_{n} \nts - \mu) * f \rvert = \lvert \mu^{}_{n} \nts
  - \mu \rvert * \lvert f \rvert$
  by \cite[Lemma~5.8.3]{Nicu-book}. If $f$ is chosen such that
  $1^{}_{K} \leqslant f \leqslant 1^{}_{U}$, one can show by standard
  estimates that
  $\| \mu^{}_{n} \nts - \mu \|^{}_{K} \leqslant \| (\mu^{}_{n} \nts -
  \mu)*f \|^{}_{\infty}$.
  Together with Lemma~\ref{lem:topologies}, this implies the
  equivalence of norm convergence and convergence in the product
  topology for such measures, which can be extremely useful in the
  diffraction context, because the spectral type is preserved under
  norm convergence \cite[Thm.~8.4]{BG}.
   
  Next, one has
  $\| \mu^{}_{n} \nts - \mu \|^{}_{K} = \sup_{x\in\vL} \lvert
  \mu^{}_{n} \nts - \mu \rvert (\{ x \})$
  under the same assumptions on the measures $\mu$ and $\mu^{}_{n}$,
  and norm convergence, and hence product convergence, is equivalent
  to \emph{uniform} convergence
  $\mu^{}_{n} (\{ x \}) \xrightarrow{\, n\to\infty\,} \mu (\{ x \})$
  for $x\in\vL$. In contrast, in this setting, simple pointwise
  convergence for $x\in\vL$ is only equivalent to vague convergence,
  $\mu^{}_{n} \rightarrow \mu$.
   
  Finally, for $f\in C_{\mathsf{c}} (\RR^d)$, one has
  $\overline{M} \bigl( (\mu^{}_{n} \nts - \mu) * f\bigr) \leqslant \|
  f \|^{}_{1} \, \overline{M} (\mu^{}_{n} \nts - \mu)$, where
\[
     \overline{M} (\mu) \, := \, \limsup_{n\to\infty} \frac{\lvert \mu 
       \rvert (A_n)}{\vol (A_n)} \ts .
\]
If $\supp (f) \subset U$, one gets equality in the previous estimate,
which shows that $\mu^{}_{n} \rightsquigarrow \mu$ is equivalent with
$\overline{M} (\mu^{}_{n} \nts - \mu) \rightarrow 0$.  \exend
\end{remark}

Behind all these relations, maybe somewhat implicitly, is the
observation that the structure of the autocorrelation imposes a
natural topology on the dynamical system itself, as explained
in more detail in \cite{MS,Bob-tale,bm}. The key result of this 
section can now be stated as follows. 

\begin{theorem}
  Let\/ $\mu^{}_{n}$ with\/ $n\in\NN$ and\/ $\mu$ be equi-translation
  bounded measures, and assume that their autocorrelations\/
  $\gamma^{}_{n} = \mu^{}_{n} \nts \circledast \widetilde{\mu^{}_{n}}$
  and\/ $\gamma = \mu \circledast \widetilde{\mu}$ exist for all\/
  $n\in\NN$.  Assume further that also the Eberlein convolutions\/
  $\mu^{}_{n} \nts \circledast \mu$ exist.
  
  Then, if\/ $\mu^{}_{n} \nts \rightsquigarrow \mu$, one also has\/
  $\gamma^{}_{n} \nts \stackrel{\pi\,}{\rightarrow} \gamma$. In
  particular, one then has\/
  $\,\lim_{n\to\infty} \gamma^{}_{n} = \ts \gamma$ in the vague
  topology.
\end{theorem}

\begin{proof}
  Let $f,g\in C_{\mathsf{c}} (\RR^d)$ be arbitrary, but fixed. Due to
  the assumed equi-translation boundedness, there is a constant $c$,
  which may depend on $f$ and $g$, so that
  $\| \mu * \nts f \|^{}_{\infty} \leqslant c$ together with
  $\| \mu^{}_{n} \nts * g \|^{}_{\infty} \leqslant c$ for all
  $n\in\NN$. 

Now, we can estimate as follows,
\[
\begin{split}
  \big\| \gamma * f * \widetilde{g} - \gamma^{}_{n} \nts 
     * f * \widetilde{g} \ts\ts \big\|_{\infty} \, & = \,
  \big\| (\mu*\nts f) \circledast \widetilde{(\mu*g)} -
    (\mu^{}_{n} \nts * \nts f) \circledast \widetilde{(\mu^{}_{n} 
    \nts * g)} \big\|_{\infty} \\[2mm]
  & \leqslant \,
   \big\| (\mu*\nts f) \circledast \widetilde{(\mu*g)} -
    (\mu * \nts f) \circledast \widetilde{(\mu^{}_{n} 
    \nts * g)} \big\|_{\infty} \\[1mm]
  & \quad\; + \big\| (\mu*\nts f) \circledast 
       \widetilde{(\mu^{}_{n} \nts * g)} -
    (\mu^{}_{n} \nts * \nts f) \circledast 
   \widetilde{(\mu^{}_{n}  \nts * g)} \big\|_{\infty} \\[2mm]
 & \leqslant \, \overline{M} \bigl( \widetilde{\mu * g} - 
    \widetilde{\mu^{}_{n} \nts * g} \bigr)
   \, \big\| \mu * \nts f \ts \big\|_{\infty} +
   \overline{M} \bigl( \mu * \nts f - 
       \mu^{}_{n} \nts * \nts f\bigr) \,
   \big\| \mu^{}_{n} \nts * g \ts \big\|_{\infty}\\[2mm]
  & \leqslant \, c \bigl(\ts
   \overline{M} ( (\mu - \mu^{}_{n}) * g) +
   \overline{M} ( (\mu - \mu^{}_{n}) * \nts f)\bigr),
\end{split}
\]
where Eq.~\eqref{eq:mean-est} was used in the 
penultimate step.

This way, we get
$\gamma^{}_{n} \nts \stackrel{\pi\,}{\rightarrow} \gamma$,
and hence $\gamma^{}_{n} \rightarrow \gamma$, from
$\mu^{}_{n} \nts \rightsquigarrow \mu$ in conjunction with the 
well-known fact
that linear combinations of functions of the form
$f \nts * \widetilde{g}$, with $f,g \in C_{\mathsf{c}} (\RR^d)$, are
dense in $C_{\mathsf{c}} (\RR^d)$.
\end{proof}

In particular, we see that we do not need vague convergence of
$\mu^{}_{n}$ to $\mu$, but rather convergence in the mean, which is
both weaker and stronger in some sense.

\subsection{Approximation by bounded measures}

Now, let us consider a sequence $(\mu^{}_n )^{}_{n\in\NN}$ of
\emph{finite} measures such that
$\lim_{n\to\infty} \mu^{}_{n} = \omega$ holds in the vague topology,
where $\omega$ is some fixed translation bounded measure, with
autocorrelation
$\gamma^{}_{\omega} = \omega \circledast \widetilde{\omega}$ relative
to a given van Hove averaging sequence $\cA$ as before. Let us assume
that the sets $A_n$ are chosen such that they can also serve as
supporting sets for the $\mu^{}_{n}$. Now, our previous question can
be rephrased as follows: When is it true that
\[
    \gamma^{}_{\omega} \, =
    \lim_{n\to\infty} \frac{\mu^{}_{n}  \nts *
       \widetilde{\mu^{}_{n}}}{\vol (A_n)} 
\]
holds? The difference to before is that we cannot define
autocorrelations for the finite measures $\mu^{}_{n}$ along
$\cA$. However, under our assumptions, we may compare
$\mu^{}_{n} \nts * \widetilde{\mu^{}_{n}}$ with
$\omega^{}_{n} \nts * \widetilde{\omega^{}_{n}}$, where
$\omega^{}_{n} := \omega |^{}_{A_n}$ is the restriction of $\omega$ to
the set $A_n$. Note that, in general, $\mu^{}_{n} \ne \omega^{}_{n}$,
and our task is to control the difference
 \[
     \myfrac{1}{\vol (A_n)} \ts 
      \bigl( \mu^{}_{n} \nts * \widetilde{\mu^{}_{n}}
     - \omega^{}_{n} \nts * \widetilde{\omega^{}_{n}} \bigr) \, = \,
     \myfrac{1}{\vol (A_n)} \ts \bigl( \mu^{}_{n} \nts * 
     ( \widetilde{\mu^{}_{n}} - \widetilde{\omega^{}_{n}})
     + (\mu^{}_{n} - \omega^{}_{n}) * \widetilde{\omega^{}_{n}} \bigr) .
 \]

 Here, we need a modified concept of mean convergence. Assume for
 simplicity that the $A_m$ are nice compact sets, say convex, such
 that $\mu^{}_{n}$ agrees with its restriction to $A_m$ for all
 $m\geqslant n$, (but not for smaller $m$). Then, with
 $\nu^{}_{n} = \mu^{}_{n} \nts - \omega^{}_{n}$, we can say that
 $\nu^{}_{n}$, with $\supp (\nu^{}_{n}) \subseteq A_n$, converges
 \emph{in mean} to $0$, denoted by $\nu^{}_{n} \rightsquigarrow 0$ in
 analogy to above, if
\begin{equation}\label{eq:finite-mean-conv}
      \limsup_{n\to\infty} \,  
    \myfrac{1}  {\vol (A_n)} \int_{K+A_n} \big|
    \nu^{}_{n} \nts * g  \big|
    (x) \dd x   \, = \, 0
\end{equation}
holds for every $g\in C_{\mathsf{c}} (\RR^d)$, with $K = K_g$ denoting
the compact support of $g$. Note that this definition makes sense
because, for any given $g$ (and hence $K$), the van Hove property of
$\cA$ implies that $\vol (K+A_n) = \cO \bigl( \vol (A_n) \bigr)$ for
large $n$. Also, one could equally well use $\lim_{n\to\infty}$
instead of $\limsup_{n\to\infty}$ in Eq.~\eqref{eq:finite-mean-conv}.

\begin{remark}
Via a standard estimate in conjunction with a
Fubini-type argument, one can see that, for any
$g\in C_{\mathsf{c}} (\RR^d)$, Eq.~\eqref{eq:finite-mean-conv} follows
from
\begin{equation}\label{eq:alt}
   \lim_{n\to\infty} 
   \frac{\lvert \nu^{}_{n} \rvert (A_n)}{\vol (A_n)}
   \, = \, 0 \ts ,
\end{equation}
which looks perhaps like a more natural way to define mean 
convergence to $0$.  However, this is a generally stronger notion, 
wherefore we prefer to use the above version.

When the measures $\nu^{}_{n}$ are equi-translation bounded, one
can show that the condition in \eqref{eq:finite-mean-conv} is
equivalent to
\[
     \lim_{n\to\infty} \myfrac{1}{\vol (A_n)} \int_{A_n}
     \big\lvert \nu^{}_{n} \nts * g \big\rvert (x) \dd x \, = \, 0 \ts ,
\]
which can be proved on the basis of the van Hove property of $\cA$.
If, in addition, all $\nu^{}_{n}$ are supported in a uniformly
discrete set $\vL \subset \RR^d$, the conditions from
Eq.~\eqref{eq:finite-mean-conv} become equivalent to
Eq.~\eqref{eq:alt}; for details, see \cite{future}.  \exend
\end{remark}

The crucial point for the approach with \emph{finite} measures is
the mutual adjustment of the supports of the approximating measures
with the elements of the van Hove sequence $\cA$. In practice, this
is usually done by selecting $\cA$ according to the approximating
measures, which often originate naturally, for instance from an
inflation rule or a similar process.

The main result of this section now reads as follows.

\begin{theorem}
  Let\/ $\omega\in\cM^{\infty} (\RR^d)$, and assume that its
  autocorrelation, $\gamma^{}_{\omega}$, exists for a given van Hove
  sequence\/ $\cA$, so
\[
   \gamma^{}_{\omega} \, = \, \omega \circledast
   \widetilde{\omega} \, = \lim_{n\to\infty}
   \frac{\omega^{}_{n} \nts * \ts \widetilde{\omega^{}_{n}}}{\vol (A_n)} 
\]
with\/ $\omega^{}_{n} := \omega |^{}_{A_n}$.  Let\/
$(\mu^{}_{n})^{}_{n\in\NN}$ be a sequence of finite measures with\/
$\supp (\mu^{}_{n}) \subseteq A_n$, and assume that the\/ $\mu^{}_{n}$
are equi-translation bounded and satisfy\/
$(\mu^{}_{n} \nts - \omega^{}_{n}) \rightsquigarrow 0$ as\/
$n\to\infty$.

Then,
$\frac{1}{\vol (A_n)} \bigl( \mu^{}_{n} \nts * \widetilde{\mu^{}_{n}}
- \omega^{}_{n} \nts * \widetilde{\omega^{}_{n}} \ts \bigr)
\stackrel{\pi\,}{\rightarrow} 0$
as\/ $n\to\infty$.  In particular, one has
\[
   \lim_{n\to\infty} \frac{\mu^{}_{n} \nts * 
  \widetilde{\mu^{}_{n}} }{\vol (A_n)}
   \, = \, \gamma^{}_{\omega}
\]
 in the vague topology. 
\end{theorem}

\begin{proof}
  Since $\omega$ is translation bounded, the finite measures
  $\omega^{}_{n}$ are equi-translation bounded. Now, given
  $f,g \in C_{\mathsf{c}} (\RR^d)$, the assumed equi-translation
  boundedness of the measures $\mu^{}_n$ implies the existence of a
  constant $c$, which may depend on $f$ and $g$, such that
  $\| \mu^{}_{n} \nts * f \|^{}_{\infty} \leqslant c$ and
  $\| \omega^{}_{n} \nts * g \|^{}_{\infty} \leqslant c$ for all
  $n\in\NN$.

  Now, with $\| \widetilde{h}\ts \|^{}_{\infty} = \| h \|^{}_{\infty}$
  for any $h\in C_{\mathsf{c}} (\RR^d)$, we get
\begin{align*}
 \lefteqn{ \big\| \mu^{}_{n} \nts * \widetilde{\mu^{}_{n}} * 
   \nts f \nts * \widetilde{g}
  -  \omega^{}_{n} \nts * \widetilde{\omega^{}_{n}} * f 
  * \widetilde{g} \ts \big\|_{\infty} } \\[3mm]
  & \quad \leqslant \, \big\| (\mu^{}_{n} \nts * \nts f) *
    \bigl( \widetilde{\mu^{}_{n}} \nts  - 
    \widetilde{\omega^{}_{n}} \bigr) * \widetilde{g} \ts \big\|_{\infty} +
   \big\| \bigl(\mu^{}_{n} \nts\nts - \omega^{}_{n} \bigr) 
   * \nts f \nts * \bigl( \widetilde{ \omega^{}_{n} \nts * g} 
   \bigr) \big\|_{\infty}  \\[3mm]
  & \quad \leqslant \, \big\| \mu^{}_{n} \nts * f \big\|_{\infty} 
   \int_{\RR^d} \bigl| \bigl( \widetilde{\mu^{}_{n}} \nts  - 
    \widetilde{\omega^{}_{n}} \bigr) * \widetilde{g}\ts \bigr|
    (x) \dd x \, + \, \big\|\ts \omega^{}_{n} \nts * g \big\|_{\infty} 
  \int_{\RR^d} \bigl| \bigl( {\mu^{}_{n}} \nts  - 
    {\omega^{}_{n}} \bigr) * \nts f \ts \bigr| (x) \dd x \\[1mm]
  & \quad \leqslant \, c \int_{K + A_n}
    \bigl| \bigl( {\mu^{}_{n}} \nts  - 
    {\omega^{}_{n}} \bigr) * \nts f \ts \bigr| (x) +
    \bigl| \bigl( {\mu^{}_{n}} \nts  - 
    {\omega^{}_{n}} \bigr) * \nts g \ts \bigr| (x) \dd x \ts ,
\end{align*}
where the compact set $K$ is chosen such that it contains the supports
of $f$ and $g$. Dividing both sides by $\vol (A_n)$, the claim follows
from the assumption of mean convergence and
Eq.~\eqref{eq:finite-mean-conv}, in complete analogy to our previous
result.
\end{proof}

The main difference to the arguments used in the main text is the
replacement of a uniform condition by a condition in mean, which
should be useful under more general circumstances.

\section*{Acknowledgements} 

It is a pleasure to thank Philipp Gohlke, Uwe Grimm, Gerhard Keller
 and Dan Rust for
helpful discussions. We thank an anonymous reviewer for a number of
suggestions that helped to improve the manuscript.
This work was supported by the German Research
Council (DFG), within the CRC 1283, and by Natural Sciences and
Engineering Council of Canada (NSERC), via grant 03762-2014. Moreover,
a research stay of N.S.\ was partially supported by the Simons
Foundation and by the Oberwolfach Research Institute for Mathematics
(MFO).

\end{document}